\documentclass[11pt]{article}

\usepackage[T1]{fontenc}
\usepackage[utf8]{inputenc}
\usepackage{amsmath,amssymb,amsthm}
\usepackage{mathrsfs}
\usepackage{bm}
\usepackage{booktabs}
\usepackage{geometry}
\usepackage{float}
\usepackage{hyperref}
\hypersetup{hidelinks}
\usepackage{cleveref}
\usepackage{enumitem}
\usepackage{graphicx}

\newtheorem{theorem}{Theorem}[section]
\newtheorem{lemma}[theorem]{Lemma}

\theoremstyle{definition}

\newtheorem{remark}[theorem]{Remark}
\newtheorem{assumption}[theorem]{Assumption}

\newcommand{\R}{\mathbb{R}}

\newcommand{\Tcal}{\mathcal{T}}
\newcommand{\Fcal}{\mathcal{F}}

\newcommand{\PiCR}{\Pi_{\rm CR}}

\newcommand{\epsh}{\varepsilon_h}
\newcommand{\jump}[1]{[\![#1]\!]}
\newcommand{\norm}[1]{\left\|#1\right\|}
\newcommand{\seminorm}[1]{\left|#1\right|}
\newcommand{\tnorm}[1]{|\!|\!|#1|\!|\!|}
\newcommand{\norme}[1]{|\!|\!|#1|\!|\!|}

\newcommand{\eps}{\varepsilon}

\title{A Stabilized CR–CR Method for Fourth-Order Elliptic Problems}
\author{
Yule Sun\thanks{
School of Mathematics and Computational Science,
Xiangtan University, Xiangtan 411105, China.
Email: \texttt{sunyuleqn@foxmail.com}.
}
\and
Shudan Tian\thanks{
School of Mathematics and Computational Science,
Xiangtan University, Xiangtan 411105, China.
Email: \texttt{shudan.tian@xtu.edu.cn}.
}
\and
Qian Zhang\thanks{
School of Mathematics, Jilin University,
Changchun 130012, China.
Email: \texttt{qzhang25@jlu.edu.cn}.
}
}
\date{}

\begin{document}
	\maketitle
	
	\begin{abstract}
		We present a CR–CR method for strain gradient elasticity and prove its robustness with respect to the Lamé parameter \(\lambda\) and the length scale \(\iota\). The method introduces the strain as an independent symmetric tensor and discretizes both the strain and the displacement using CR elements. The commuting properties of the CR interpolation operator with the relevant differential operators leads the robustness with respect to $\lambda$. A stabilization term weakly imposes the relation between the strain and the symmetric gradient of the displacement. The method also applies directly to the biharmonic equation and singularly perturbed fourth-order problems. For these problems, we prove that the CR–CR method is equivalent to an RT–CR method proposed by Huang and Tang. This result extends the classical equivalence between the RT–\(P_0\) mixed method and the CR method for the Poisson equation to fourth-order problems. The equivalence shows that the stabilization term arises naturally from the RT–CR mixed formulation.
	\end{abstract}
	
	
	\section{Introduction}
	\label{sec:introduction}
	Let \(\Omega\subset\mathbb R^d\), \(d=2,3\), be a connected bounded domain. In this article, we
	consider the linear strain-gradient elasticity(SGE) problem proposed by \cite{aifantis1992role}.
	\begin{equation}
		\label{eq:sge-model}
		\iota^2\operatorname{div}\Delta\sigma(u)
		-\operatorname{div}\sigma(u)=f
		\quad\text{in }\Omega,
		\qquad
		u=\partial_n u=0
		\quad\text{on }\partial\Omega,
	\end{equation}
	where
	\[
	\sigma(u)=2\mu\varepsilon(u)+\lambda(\operatorname{div}u)I,
	\qquad
	\varepsilon(u)=\frac12\bigl(\nabla u+\nabla u^T\bigr).
	\]
	Here \(0<\iota\le 1\) is the material length-scale parameter, 
	\(0<\mu_0 \le \mu\le \mu_1<\infty\) and \(0\le \lambda\) are the Lam\'e parameters. The SGE model with a single length scale parameter can eliminate stress concentrations and capture size effects~\cite{AskesAifantis2011}. More details of the singular parameter SGE model can be found in \cite{AskesAifantis2011}. 

    One of the main contributions of this work is the development of a robust piecewise linear finite element method for problem \eqref{eq:sge-model} based on the Crouzeix--Raviart (CR) element. And this method needn't any penalty parameters. 
    Designing a low-order finite element method for SGE is challenging.
    Robustness with respect to $\lambda$ requires approximation properties for both the divergence of the displacement and its gradient. This implies a direct discretization using standard $C^0$ piecewise quadratic Lagrange elements \cite{BrennerSung2005} is failed. Robustness with respect to $\iota$ requires the finite element method to handle both second-order and fourth-order problems, which means that we can not directly use Morley element\cite{Morley1968,Trygve2001}. Clearly, the CR element cannot be used directly to solve the SGE problem either.
Therefore, define a auxiliary variable $\phi=\epsilon(u)$ and propose the following discrete formulation:
find
$
(u_h,\phi_h)\in
[\mathrm{CR}_{h,0}]^d
\times
[\mathrm{CR}_{h,0}]^{d\times d}_{\mathrm{sym}}
$
such that, for all $(v_h,\psi_h)\in [\mathrm{CR}_{h,0}]^d
\times
[\mathrm{CR}_{h,0}]^{d\times d}_{\mathrm{sym}}$,
\[
\iota^2
(\nabla_h(\mathcal{C}\phi_h),\nabla_h\psi_h)
+
(\mathcal{C}\epsilon_h(u_h),\epsilon_h(v_h))
+
s^{sg}_h((\phi_h,u_h),(\psi_h,v_h))
=
(f,v_h).
\]
Here, $\mathcal{C}$ is the elasticity tensor, defined by
$\mathcal{C}\tau=2\mu\tau+\lambda\operatorname{tr}(\tau)I$
for symmetric tensors $\tau$.
The space $\mathrm{CR}_{h,0}$ is the CR space with zero mean values
on boundary faces, and stabilization term is defined as
\[
\begin{aligned}
s^{sg}_h((\phi_h,u_h),(\psi_h,v_h)):=&
\iota^2\sum_{T\in\mathcal{T}_h}h_T^{-2}
\bigl(
\epsilon_h(u_h)-Q_0\phi_h,\,
\epsilon_h(v_h)-Q_0\psi_h
\bigr)_T
+
\sum_{F\in\mathcal{F}_h^i}h_F^{-1}
\langle [u_h],[v_h]\rangle_F.
\end{aligned}
\]
Here $Q_0$ the $L^2$-orthogonal projection operator. The first term in the stabilization weakly enforces
$\phi=\epsilon(u)$, while the second term is used to ensure
a discrete Korn inequality. 
The same approach can also be applied to the biharmonic equation
and the singularly perturbed fourth-order problem.
In both cases, the jump term can be omitted.
For the biharmonic equation, we introduce $\phi=\nabla u$ and seek
$
(u_h,\phi_h)\in
\mathrm{CR}_{h,0}\times[\mathrm{CR}_{h,0}]^d
$
such that, for all
$(v_h,\psi_h)\in
\mathrm{CR}_{h,0}\times[\mathrm{CR}_{h,0}]^d$,
\[
(\nabla_h\phi_h,\nabla_h\psi_h)
+
s_h^{\mathrm{bi}}((\phi_h,u_h),(\psi_h,v_h))
=
(f,v_h),
\]
where
\[
s_h^{\mathrm{bi}}((\phi_h,u_h),(\psi_h,v_h))
:=
\sum_{T\in\mathcal{T}_h}h_T^{-2}
\bigl(
\nabla_hu_h-Q_0\phi_h,\,
\nabla_hv_h-Q_0\psi_h
\bigr)_T.
\]
The CR--CR formulation for the biharmonic equation is similar in form
to the Reissner--Mindlin plate formulation.
The factor $h_T^{-2}$ in the stabilization plays a role analogous
to the factor $t^{-2}$ in the shear term of the normalized
Reissner--Mindlin model.
In the CR--CR method, the finite element solution converges
to the solution of the biharmonic equation as the mesh size $h$
tends to zero.
In the Reissner--Mindlin model, the solution approaches that
of the Kirchhoff--Love plate model as the plate thickness $t$
tends to zero.
Both limits recover the constraint $\phi=\nabla u$,
although $h$ is a discretization parameter and $t$ is a physical parameter.

The idea of introducing $\phi=\nabla u$ to reduce the order
of the biharmonic equation is not new.
In 2017, Gallistl~\cite{Gallistl2017} established a stable splitting
of elliptic problems of order $2m$ into one generalized Stokes problem
and $2(m-1)$ Poisson-type problems.
Chen, Pani, and Qiu~\cite{ChenPaniQiu2022} proposed a stabilized
mixed finite element method using $\mathrm{div}$-conforming piecewise polynomials of degree
$k+1$ for $\phi$ and continuous piecewise polynomials of degree
$k+2$ for $u$, with $k\ge0$.
Their formulation involves a penalty parameter.
In the lowest-order case with a constant coefficient,
the penalty parameter can be fixed at $1$. 

Another contribution of this work is to derive explicit formulas
that allow the numerical solutions of the proposed method and
the mixed method recently introduced by Huang and
Tang~\cite{HuangTang2025} to be reconstructed from each other.
These formulas extend Marini's 
formula in \cite{Marini1985}, which relates the CR and
$\mathrm{RT}_0$--$ P_0$ solutions of the Poisson problem,
to fourth-order problems. For the SGE problem, the proposed method is not exactly equivalent
to the recent method of Huang and Tang~\cite{HuangTang2026SGE}
unless the stabilization term is modified. Their method
corresponds to a more complicated form of stabilization.
   
	Several finite element methods have also been developed for
	strain-gradient elasticity.  Early approaches include continuous/discontinuous
	Galerkin approximations and mixed formulations
	\cite{engel2002continuous,amanatidou2002mixed}. Until 2017, there some numerical analysis for SGE problem and proposed some $\iota$-robust nonconforming
	\(H^2\)-type element \cite{li2017two,liao2019family,li2021h2korn}.  Mixed methods avoid globally \(H^2\)-conforming
	displacement spaces by introducing additional variables.  This direction
	includes a robust lower-order mixed method, Taylor--Hood type methods for the
	nearly incompressible regime, and efficient three-dimensional formulations
	\cite{chen2023robust,liao2023taylor,riesselmann2024efficient}.  More recent
	alternatives employ Nitsche's technique, \(C^0\) interior penalties,
	nonconforming virtual elements, or isogeometric spaces
	\cite{chen2025nitsche,huang2025optimal,chen2026robust,
		huang2026c0vem,wriggers2023virtual,hosseini20223d}.

	\section{Notation, CR spaces, and projections}
	\label{sec:preliminaries}
	
	Given a bounded domain \(G \subset \mathbb R^d\) and a non-negative integer \(m\), let \(H^m(G)\) be the usual Sobolev space of functions on \(G\), and \(H^m(G;\mathbb K):=H^m(G)\otimes\mathbb K\), where \(\mathbb{K}\in \{\mathbb R,\,\mathbb R^d,\,\mathbb R^{d\times d},\, \mathbb R^{d\times d\times d}\}\). The corresponding norm and semi-norm are denoted respectively by \(\|\cdot\|_{m,G}\) and \(|\cdot|_{m,G}\). Let \((\cdot,\cdot)_G\) be the standard inner product on \(L^2(G;\mathbb K):=L^2(G)\otimes\mathbb K\). For a positive number \(s\) that is not an integer, \(H^s(G;\mathbb K)\) is the fractional order Sobolev space. Let \(m=\lfloor s\rfloor\) be the largest integer less than \(s\). The semi-norm \(|v|_{s,G}\) and the norm \(\|v\|_{s,G}\) are respectively given by
	\begin{align*}
		|v|_{s,G}^2
		&=
		\sum_{|\alpha|=m}
		\int_{G}\int_{G}
		\frac{
			\left|(\partial^\alpha v)(x)-(\partial^\alpha v)(y)\right|^2
		}{
			|x-y|^{d+2(s-m)}
		}
		\,\mathrm{d}x\,\mathrm{d}y,
		\\
		\|v\|_{s,G}^2
		&=
		\|v\|_{m,G}^2
		+
		|v|_{s,G}^2.
	\end{align*}
	If \(G\) is \(\Omega\), we abbreviate \(\|\cdot\|_{m,G}\), \(|\cdot|_{m,G}\) and \((\cdot,\cdot)_{G}\) by \(\|\cdot\|_m\), \(|\cdot|_m\) and \((\cdot,\cdot)\), for all real orders \(m\ge 0\).
	
	Let $\mathcal T_h$ be a shape-regular simplicial partition of $\Omega$ with maximum mesh size $h=\max_{T\in\mathcal T_h}h_T$ and $h_T:=\operatorname{diam}(T)$.  Let
	\(\mathcal F_h=\mathcal F_h^i\cup\mathcal F_h^\partial\)
	be the set of interior and boundary faces of the triangulation, and let
	\(\mathcal E_h=\mathcal E_h^i\cup\mathcal E_h^\partial\)
	be the corresponding decomposition of its
	\((d-2)\)-dimensional subsimplices. 
	When \(d=2\), \(\mathcal E_h\) is precisely the set of mesh vertices. The broken gradient \(\nabla_h\) is defined as the elementwise counterpart of \(\nabla\) with respect to \(\mathcal T_h\).
	
	For each interior face
	\(F=T^+\cap T^-\in\mathcal F_h^i\), fix the unit normal
	\(n_F\) pointing from \(T^+\) to \(T^-\). The jump and average of a piecewise function
	\(v\) across \(F\) are defined by
	\[
	\jump{v}:=v|_{T^+}-v|_{T^-},
	\qquad
	\{v\}_F:=\frac12\left(v|_{T^+}+v|_{T^-}\right).
	\]
	For a boundary face \(F\in\mathcal F_h^\partial\), \(n_F\) is the outward
	unit normal, and we set
	\[
	\jump{v}:=v|_F,
	\qquad
	\{v\}_F:=v|_F.
	\]
	These conventions are applied componentwise to vector- and
	tensor-valued functions. 
	
	We use \(\mathrm{CR}_h\) to denote the lowest-order Crouzeix--Raviart element on the triangulation \(\mathcal{T}_h\); see \cite[Example~4]{crouzeix1973conforming}. 
	For an integer \(k\ge 0\), let $\mathbb P_k(T)$ denote the space of polynomials of degree at most $k$ on $T$. The Crouzeix--Raviart space is defined as follows: For any \(T\in\mathcal T_h\), the local shape function space is \(\mathbb P_1(T)\) and the nodal variables are given by
	\begin{equation}
		\label{eq:cr-dofs}
		\left\{\frac{1}{|F|}\int_F v\mathrm{d}s,\qquad \forall F\in \mathcal F_h,~F\subset \partial T\right\}.
	\end{equation}
	The global space is defined by
	\begin{equation}
		\label{eq:cr-global}
		\mathrm{CR}_h:=\left\{
		v\in \mathbb P_1(\mathcal T_h):\text{\eqref{eq:cr-dofs} 
			is single-valued on all } F\in \mathcal F^i_h \right\}.
	\end{equation}
	And define
	\begin{equation}
		\mathrm{CR}_{h0}=\left\{
		v\in \mathrm{CR}_h:\,\int_F v\mathrm{d}s=0,\quad \forall F\in \mathcal F^{\partial}_h\right\}.
	\end{equation}
	
	Let $Q_k:L^2(\Omega;\mathbb K)\to \mathbb P_k(\mathcal T_h)\otimes\mathbb K$ be the the elementwise $L^2$ projection, where \(\mathbb P_k(\mathcal T_h):=\prod_{T\in\mathcal T_h}\mathbb P_k(T)\).
	Let $\Pi_{\rm CR}:H^1(\Omega;\mathbb K)\to {\rm{CR}}_h\otimes \mathbb K$ be the CR interpolation operator, defined by
	\[
	\frac{1}{|F|}\int_F \Pi_{\rm CR} v\mathrm{d}s= \frac{1}{|F|}\int_F v\mathrm{d}s,\qquad \forall F\in\mathcal F_h,\]
    for all \(v\in H^1(\Omega,\mathbb K)\).
	
	\begin{lemma}
		\label{lem:cr-commuting-general}
		For $v\in H^1(\Omega)$ and $\psi\in H^1(\Omega;\R^d)$,
		\begin{equation}
			\label{eq:cr-commuting-general}
			\nabla_h(\PiCR v)=Q_0(\nabla v),
			\qquad
			\nabla_h(\PiCR \psi)=Q_0(\nabla \psi) .
		\end{equation}
		Define \(\varepsilon_h(\psi)=(\nabla_h \psi+(\nabla_h \psi)^\intercal)/2\) and \(\operatorname{div}_h:=\nabla_h\cdot\). Taking the symmetric part and the trace of the matrix identity yields
		\begin{equation}
			\label{eq:cr-commuting-sym}
			\epsh(\PiCR \psi)=Q_0(\eps(\psi)),
			\qquad
			\operatorname{div}_h(\Pi_{\rm CR}\psi)=Q_0(\operatorname{div} \psi).
		\end{equation}
	\end{lemma}
	
	\begin{proof}
		For the scalar identity, for \(T\in \mathcal T_h\)
		\[
		\int_T \nabla(\PiCR v)\mathrm{d}x
		=\int_{\partial T}\PiCR v\, n\,\mathrm{d}s
		=\int_{\partial T}v\, n\,\mathrm{d}s
		=\int_T\nabla v\,\mathrm{d}x.
		\]
		The vector-valued identity follows componentwise.
	\end{proof}
	The following Lemma can be obtained in \cite{{LasisSuli2003}}.
	\begin{lemma}
		\label{lem:CR-discrete-poincare}
		For all \(v_h\in\mathrm{CR}_{h,0}\),
		\begin{equation}
			\label{eq:CR-discrete-poincare}
			\norm{v_h}_0
			\lesssim
			\norm{\nabla_h v_h}_0.
		\end{equation}
		The estimate also holds componentwise for vector- and
		tensor-valued CR functions.
	\end{lemma}
	
	
	\section{The singular perturbation problem}
	\label{sec:biharmonic}
	
	In this section, we apply the CR–CR scheme for the solution of the singularly perturbed biharmonic model. Given \(f\in L^2(\Omega)\), the fourth-order elliptic singular perturbation problem is to find \(u\in H^2_0(\Omega)\) satisfying
	\begin{equation}
		\label{eq:bih-model}
		\iota^2\Delta^2u-\Delta u=f\quad\text{in }\Omega,
		\qquad
		u=\partial_nu=0\quad\text{on }\partial\Omega,
	\end{equation}
	where \(n\) is the unit outward normal to \(\partial \Omega\), and \(\iota\) is a small positive parameter. The variational formulation of problem \eqref{eq:bih-model} is to find 
	\(u\in H_0^2(\Omega)\) such that
	\begin{equation}
		\label{eq:bih-weak}
		\iota^2(\nabla^2u,\nabla^2v)+(\nabla u,\nabla v)=(f,v)
		\qquad\forall v\in H_0^2(\Omega).
	\end{equation}
	
	When \(\iota=0\), the fourth-order elliptic singular perturbation problem \eqref{eq:bih-model} reduces to the Poisson equation
	\begin{equation}
		\label{eq:Poisson-model}
		-\Delta u_0=f\quad\text{in }\Omega,
		\qquad
		u_0=0\quad\text{on }\partial\Omega.
	\end{equation}
	\begin{assumption}
		\label{assum:Regularity}
		Assume that the Poisson equation \eqref{eq:Poisson-model} satisfies the regularity
		\begin{equation}
			\label{eq:Poisson-regularity}
			\|u_0\|_2\lesssim \|f\|_0,
		\end{equation}
		and that the singularly perturbed fourth-order problem \eqref{eq:bih-model} satisfies the
		regularity
		\begin{equation}
			\label{eq:bib-regularity}
			\|u-u_0\|_1+\iota\|u\|_2+\iota^2\|u\|_3\lesssim\iota^{1/2}\|f\|_0.
		\end{equation}
	\end{assumption}
	When \(\Omega\) is convex, we refer to \cite{Gao2020} for the regularity \eqref{eq:Poisson-regularity} and to \cite[Lemma 5.1]{Trygve2001} and \cite[Lemma 4]{Guzmán2012} for the regularity \eqref{eq:bib-regularity} in two and three dimensions.
	
	The finite element is defined as follows: For any \(T\in \mathcal T_h \), the local shape function space is \(\mathbb P_1(T)\times \mathbb P_1(T;\mathbb R^d)\) and the nodal variables are given by
	\begin{equation}
		\label{eq:bih-dofs}
		\left\{\frac{1}{|F|}\int_F v\mathrm{d}s,\quad\frac{1}{|F|}\int_F \psi\mathrm{d}s,\qquad \forall F\in \mathcal F_h,~F\subset \partial T\right\},
	\end{equation}
	for \((v,\psi)\in \mathbb P_1(T)\times \mathbb P_1(T,\mathbb R^d)\). Associated with the partition \(\mathcal T_h\), the global space  
	\begin{equation}
		\label{eq:bih-global-space}
		V_h^{\rm sp}:=\left\{
		(v,\psi)\in \mathbb P_1(\mathcal T_h)\times \mathbb P_1(\mathcal T_h,\mathbb R^d):
		\text{ \eqref{eq:bih-dofs} is single‑valued on all } F\in \mathcal F_h^i
		\right\}.
	\end{equation}
	And define
	\begin{equation}
		V_{h0}^{\rm sp}=\left\{
		(v,\psi)\in V_h^{\rm sp}:\,
		\frac{1}{|F|}\int_F v\mathrm{d}s=0,~\frac{1}{|F|}\int_F \psi\mathrm{d}s=0,
		\quad\forall F\in\mathcal F^\partial_h\right\}.
	\end{equation}
	
	Throughout the remainder of the paper, we use boldface symbols to denote ordered pairs. For \(\bm w_h=(w_h,p_h),\bm v_h=(v_h,\psi_h)\in V_h^{\rm sp}\), set discrete bilinear forms
	\begin{align*}
		a_h^{\rm sp}(\bm w_h,\bm v_h):=&\iota^2(\nabla_hp_h,\nabla_h\psi_h)+(\nabla_h w_h,\nabla_h v_h),\\
		s_h^{\rm sp}(\bm w_h,\bm v_h):=&\iota^2\sum_{T\in \mathcal T_h}h^{-2}_T(\nabla_h w_h-Q_0p_h,\nabla_hv_h-Q_0\psi_h)_T.
	\end{align*}
	where \(h_F:=\operatorname{diam}(F)\). Then the CR-CR method for problem \eqref{eq:bih-weak} is to find \(\bm u_h=(u_h,\phi_h)\in V_{h0}^{\rm sp}\) such that
	\begin{equation}
		\label{eq:bih-discrete}
		a_h^{\rm sp}(\bm u_h,\bm v_h)+s_h^{\rm sp}(\bm u_h,\bm v_h)=(f,v_h),\qquad\forall \bm v_h\in V_{h0}^{\rm sp}.
	\end{equation}
	For any \(\bm v_h\in V_h^{\rm sp}\), define a semi-norm 
	\begin{equation}
		\label{eq:bih-norme}
		\tnorm{\bm v_h}^2:=a_h^{\rm sp}(\bm v_h,\bm v_h)+s_h^{\rm sp}(\bm v_h,\bm v_h).
	\end{equation}
	For all \(\bm v_h\in V_{h0}^{\rm sp}\), by \eqref{eq:CR-discrete-poincare}, we have 
	\begin{equation}
		\label{eq:bih-H1-control}
		\norm{v_h}_0+\norm{\nabla_h v_h}_0
		+\iota\norm{\psi_h}_0+\iota\norm{\nabla_h\psi_h}_0
		\lesssim\tnorm{\bm v_h}.
	\end{equation}
	Hence, \(\tnorm{\cdot}\) defines a norm on \(V_{h0}^{\rm sp}\). The well-posedness of CR--CR method \eqref{eq:bih-discrete} follows immediately from \eqref{eq:bih-H1-control}.
	
	For \(v\in H^2(\Omega)\), set
	\begin{equation}
		\label{eq:bih-interpolant}
		\Pi_h^{\rm sp}v:=(\PiCR v,\PiCR\nabla v)\in V_h^{\rm sp}.
	\end{equation}
	If \(v\in H^2_0(\Omega)\), then \(\Pi_h^{\rm sp} v\in V_{h0}^{\rm sp}\). Since the stabilization term within \(\tnorm{(v,\nabla v)- \Pi_h^{\rm sp} v}\) need not tend to zero as the mesh is refined,
	we introduce the semi-norm 
	\begin{equation}
		\label{eq:bih-semi-norm}
		|\bm v|^2_{\rm E}:=a_h^{\rm sp}(\bm v,\bm v),\qquad \bm v=(v,\psi)\in (H^2(\Omega)\times H^1(\Omega;\mathbb R^d))+V_h^{\rm sp},
	\end{equation}
	and consider error estimates in this semi-norm. Let \(u\) and \(\bm u_h\) be the solutions to \eqref{eq:bih-weak} and \eqref{eq:bih-discrete}, respectively.
	By the triangle inequality
	\[
	|(u,\nabla u)-\bm u_h|_{\rm E}\le |(u,\nabla u)-\Pi_h^{\rm sp} u|_{\rm E}+\tnorm{\Pi_h^{\rm sp} u- \bm u_h}.\]
	Thus, in order to derive estimates for the interpolation error \(|(u,\nabla u)-\Pi_h^{\rm sp} u|_{\rm E}\) and the stabilization in \(\tnorm{\Pi_h^{\rm sp} u- \bm u_h}\), we present the following lemma.
	\begin{lemma}
		\label{lem:bih-fractional-means}
		Let \(0<s\le1\). For every field
		\(v\in H^{s}(\Omega;\mathbb K)\),
		\begin{equation}
			\label{eq:bih-fractional-means}
			\norm{(I-Q_0) v}_0^2
			+\sum_{F\in\Fcal_h^i}h_F
			\norm{\jump{Q_0 v}
			}_{0,F}^2
			\lesssim h^{2s}\seminorm{v}_{s}^2,
		\end{equation}
		where \(I\) is the identity operator.
	\end{lemma}
	
	\begin{proof}
		Using the trace inequality, we obtain
		\[h_F\|\jump{Q_0v}\|^2_{0,F}
		=h_F\|\jump{Q_0v-Q_{\omega_F}v}\|^2_{0,F}
		\lesssim \|(I-Q_0)v\|^2_{0,\omega_F}+\|(I-Q_{\omega_F})v\|^2_{0,\omega_F},\]
		where $Q_{\omega_F}$ denotes the constant $L^2$-projection onto the patch $\omega_F$. The standard \(L^2\)-projection error estimate then gives \eqref{eq:bih-fractional-means}.
	\end{proof}
	
		%
	
	\begin{lemma}
		\label{lem:bih-convergence}
		Let \(u\in H^{2+s}(\Omega)\cap H_0^2(\Omega)\) be the solution of \eqref{eq:bih-weak}, \(0<s\le 1\). Under Assumption~\ref{assum:Regularity}, we have 
		\begin{align}
			|(u,\nabla u)-\bm u_h|_{\rm E}
			\lesssim&
			\min\Bigl\{
			\iota h^s|u|_{2+s}
			+h|u|_2,
            (\iota^{1/2}+h)\norm{f}_0\Bigr\}
			\notag\\
			&
			+\sup_{0\ne\bm v_h\in V_{h0}^{\rm sp}}\frac{a_h^{\rm sp}(\Pi_h^{\rm sp} u,\bm v_h)-(f,v_h)}{\tnorm{\bm v_h}},
			\label{eq:bih-energy-low-reg}
		\end{align}
		
	\end{lemma}
	
	\begin{proof}
		The commuting identities \eqref{eq:cr-commuting-general} imply
		\[|(u,\nabla u)-\Pi_h^{\rm sp} u|^2_{\rm E}=\iota^2\|(I-Q_0)\nabla^2 u\|^2_0+\|(I-Q_0)\nabla u\|^2_0.\] 
		For the fourth-order part, the local approximation property and the
		$L^2$-stability of $Q_0$ give
		\[
		\norm{(I-Q_0)\nabla^2u}_0
		\lesssim
		\min\left\{
		h^s|u|_{2+s},\,
		|u|_2
		\right\}.
		\]
		For the second-order part, the standard approximation estimate gives
		\[
		\norm{(I-Q_0)\nabla u}_0
		\lesssim
		h|u|_2.
		\]
		Alternatively, decomposing $u=(u-u_0)+u_0$, and using Lemma~\ref{lem:bih-fractional-means}, we obtain
		\begin{align*}
			\norm{(I-Q_0)\nabla u}_0
			&\le
			\norm{(I-Q_0)\nabla(u-u_0)}_0
			+
			\norm{(I-Q_0)\nabla u_0}_0
			\\
			&\lesssim
			|u-u_0|_1
			+h|u_0|_{2}.
		\end{align*}
		Combining \eqref{eq:bib-regularity}, \eqref{eq:Poisson-regularity} and these estimates yields
		\begin{align}
			\label{eq:perturbation-energy-interpolation}
			|(u,\nabla u)-\Pi_h^{\rm sp}u|_{\rm E}
			\lesssim{}
			\min\left\{
			\iota h^s
			|u|_{2+s}
			+h|u|_2,\,
			(\iota^{1/2}+h)\norm{f}_0
			\right\}.
		\end{align}
		
		From the discrete scheme \eqref{eq:bih-discrete}, we obtain
		\begin{equation*}
			\tnorm{\Pi_h^{\rm sp} u-\bm u_h}^2=a_h^{\rm sp}(\Pi_h^{\rm sp}u,\Pi_h^{\rm sp} u-\bm u_h)-(f,\Pi_{\rm CR} u- u_h) +s_h^{\rm sp}(\Pi_h^{\rm sp} u, \Pi_h^{\rm sp} u-\bm u_h)
		\end{equation*}
		Then we have 
		\begin{equation}
			\tnorm{\Pi_h^{\rm sp} u-\bm u_h}\le \sup_{0\ne\bm v_h\in V_{h0}^{\rm sp}}\frac{a_h^{\rm sp}(\Pi_h^{\rm sp} u,\bm v_h)-(f, v_h)}{\tnorm{\bm v_h}}
			+\sup_{0\ne\bm v_h\in V_{h0}^{\rm sp}}\frac{s_h^{\rm sp}(\Pi_h^{\rm sp} u, \bm v_h)}{\tnorm{\bm v_h}}
		\end{equation}
		Since $s_h^{\rm sp}$ is a symmetric positive‑semidefinite bilinear form,
		\begin{equation}
			|s_h^{\rm sp}(\Pi_h^{\rm sp} u, \bm v_h)|
			\le s_h^{\rm sp}(\Pi_h^{\rm sp} u, \Pi_h^{\rm sp} u)^{1/2}s_h^{\rm sp}(\bm v_h,\bm v_h)^{1/2}
			\le s_h^{\rm sp}(\Pi_h^{\rm sp} u, \Pi_h^{\rm sp} u)^{1/2}\tnorm{\bm v_h}
		\end{equation}
		Using commuting identity \eqref{eq:cr-commuting-general} and \cite[(3.2)]{cai2017Discontinuous} yields
		\begin{align}
			\label{eq:inter-stabilization}
			s_h^{\rm sp}(\Pi_h^{\rm sp} u, \Pi_h^{\rm sp} u)^{1/2}=&
			\left(\iota^2\sum_{T\in\Tcal_h}h_T^{-2}
			\norm{Q_0(\nabla u-\PiCR\nabla u)}_{0,T}^2\right)^{1/2}\notag\\
			\lesssim& \min\{h^s\iota|u|_{2+s},\iota^{1/2}\norm{f}_0\}.
		\end{align}
		Combining \eqref{eq:perturbation-energy-interpolation} and \eqref{eq:inter-stabilization} completes the proof.
	\end{proof}
	
	We next estimate the consistency term
    \[a_h^{\rm sp}(\Pi_h^{\rm sp}u,\bm v_h)-(f,v_h),\]
    assuming only \(u\in H^{2+s}(\Omega)\) and \(u_0\in H^{2}(\Omega)\), with \(0<s\le1\). For each \(\bm v_h\in V_{h0}^{\rm sp}\), we construct a Morley--Wang--Xu function, enrich it into a \(H^2(\Omega)\)-conforming finite element function, and prove the required approximation estimates.
	
	The Morley-Wang-Xu element \cite{Ming2006} is defined as follows: For any \(T\in \mathcal T_h\), the local shape function space is \(\mathbb P_2(T)\) and the nodal variables are given by
	\begin{equation}
		\label{eq:morley-dofs}
		\left\{
		\frac{1}{|E|}\int_E v\mathrm{d} s,\quad\frac{1}{|F|}\int_F \partial_{n_F} v\mathrm{d}s,\qquad \forall E\in \mathcal E_h,~F\in \mathcal F_h,~E,F\subset\partial T\right\}.
	\end{equation}
	When \(d=2\), \(\frac{1}{|E|}\int_E v\,\mathrm{d} s\) denotes the point value \(v(E)\). The global Morley-Wang-Xu element space is 
	\begin{equation}
		\label{eq:global-morley}
		M_h:=\{v\in \mathbb P_2(\mathcal T_h):\, \eqref{eq:morley-dofs}\text{ are single-valued}\}.
	\end{equation}
	And define
	\begin{equation}
		\label{eq:M0}
		M_{h0}:=\left\{
		v\in M_h:\, \frac{1}{|E|}\int_E v\mathrm{d}s=\frac{1}{|F|} \int_F\partial_{n_F} v\mathrm{d}s=0,\qquad\forall E\in\mathcal E^\partial_h,F\in \mathcal F^\partial_h\right\}.
	\end{equation}
	
	Combining this with the jump estimate below, we derive the approximation estimate for the Morley‑Wang‑Xu element.
	
	\begin{lemma}
		\label{lem:bih-strong-jump}
		For every \(\bm v_h=(v_h,\psi_h)\in V_{h0}^{\rm sp}\),
		\begin{equation}
			\label{eq:bih-strong-jump}
			\sum_{F\in\mathcal F_h}h_F^{-3}
			\|\jump{v_h}\|_{0,F}^2
			\lesssim
			\iota^{-2}\tnorm{\bm v_h}^2.
		\end{equation}
	\end{lemma}
	
	\begin{proof}
		Since \(v_h\in\mathrm{CR}_{h,0}\), its jump has zero mean on all
		faces.
		The Poincar\'e inequality therefore gives
		\[
		h_F^{-3}\|\jump{v_h}\|_{0,F}^2
		\lesssim
		h_F^{-1}\|\nabla_F\jump{v_h}\|_{0,F}^2,
		\]
		where \(\nabla_F=(I-n_F\otimes n_F)\nabla\).
		
		Since \(\nabla_h v_h\) is elementwise constant and
		\(\int_F \jump{\psi_h}\mathrm{d}s=0\),
		\[
		\nabla_F\jump{v_h}
		=
		\frac1{|F|}(I-n_F\otimes n_F)
		\int_F\jump{\nabla_h v_h-\psi_h}\,ds .
		\]
		Hence, Cauchy's inequality and the inverse trace inequality give
		\begin{align*}
			h_F^{-3}\|\jump{v_h}\|_{0,F}^2
			\lesssim
			h_F^{-1}\|\jump{\nabla_h v_h-\psi_h}\|_{0,F}^2
			\lesssim
			\sum_{T\subset\omega_F}
			h_T^{-2}
			\|\nabla_h v_h-\psi_h\|_{0,T}^2,
		\end{align*}
		where \(\omega_F\) is the union of the elements sharing the face \(F\).
		
		Finally,
		\[
		\nabla_h v_h-\psi_h
		=
		(\nabla_h v_h-Q_0\psi_h)
		+
		(Q_0\psi_h-\psi_h),
		\]
		and the local projection estimate yields
		\[
		h_F^{-3}\|\jump{v_h}\|_{0,F}^2
		\lesssim
		\sum_{T\subset\omega_F}
		\left(
		h_T^{-2}\|\nabla_h v_h-Q_0\psi_h\|_{0,T}^2
		+\|\nabla_h\psi_h\|_{0,T}^2
		\right).
		\]
		Summing over all faces and using the finite overlap of the face patches
		completes the proof.
	\end{proof}
	
	\begin{lemma}
		\label{lem:bih-morley-companion}
		Let \(d=2,3\). For every \(\bm v_h= (v_h,\psi_h)\in V_{h0}^{\rm sp}\), there exists \(m_v\in M_{h0}\) such that 
		\begin{align}
			\frac{1}{|E|}
			\int_E m_v\mathrm{d}s
			&=
			\frac1{\#\omega_E}
			\sum_{T\in\omega_E}
			\frac{1}{|E|}
			\int_E v_h|_T\mathrm{d}s,
			& \forall E\in\mathcal E_h^i,
			\label{eq:bih-morley-codimtwo-dof}\\
			\frac{1}{|F|}\int_F \partial_{n_F} m_v\mathrm{d}s
			&=
			\frac{1}{|F|}\int_F \psi_h\cdot n_F\mathrm{d}s,
			&\forall F\in\mathcal F_h^i,
			\label{eq:bih-morley-face-dof}
		\end{align}
		where \(\omega_E\) is the element patch consisting of all \(T\in \mathcal T_h\)
		that contain \(E\), and \(\#\omega_E\) denotes the number of elements
		in this patch. Then, we have 
		\begin{equation}
			\label{eq:lifting-estimate}
			\sum_{T\in \mathcal T_h }\left(h^{-4}_T\|m_v-v_h\|_{0,T}^2+h^{-2}_T\|\nabla_hm_v-\psi_h\|^2_{0,T}+\|\nabla^2_hm_v\|^2_{0,T}\right)
			\lesssim\iota^{-2}\tnorm{\bm v_h}^2.
		\end{equation}
	\end{lemma}
	
	\begin{proof}
		By the scaled norm equivalence for the
		Morley degrees of freedom,
		\begin{align}
			\sum_{T\in \mathcal T_h}h^{-4}_T\|m_v-v_h\|^2_{0,T}
			\lesssim
			\sum_{T\in \mathcal T_h}
			\sum_{E\in \mathcal E_h\cap \partial T}h^{d-4}_T
			\left|\frac{1}{|E|}\int_E m_v-v_h|_T\mathrm{d}s\right|^2\notag\\
			+\sum_{T\in \mathcal T_h}
			\sum_{F\in\mathcal F_h\cap\partial T}
			h^{d-2}_T\left|
			\frac{1}{|F|}\int_F (\psi_h-\nabla_hv_h|_T)\cdot n_F\mathrm{d}s\right|^2.
		\end{align}
		
		By the definition of the degree of freedom \eqref{eq:bih-morley-codimtwo-dof}, the difference between \(\frac{1}{|E|}\int_E m_v\mathrm{d}s\) and any element trace \(\frac{1}{|E|}\int_E v_h|_T\mathrm {d}s\) is a finite sum of jumps of \(v_h\) across faces containing \(E\). When \(d=3\),
		\begin{align*}
			\left|\frac{1}{|E|}\int_E m_v-v_h|_T\mathrm{d}s\right|^2
			\lesssim &\sum_{F\in \mathcal F_h, E\subset F}
			\left|\frac{1}{|E|}\int_E \jump{v_h}\mathrm{d}s\right|^2\\
			\lesssim&\sum_{F\in \mathcal F_h, E\subset F}
			h^{-1}_F\|\jump{v_h}\|^2_{0,E}
			\lesssim 
			\sum_{F\in \mathcal F_h, E\subset F}
			h^{-2}_F
			\|\jump{v_h}\|^2_{0,F}.
		\end{align*}
		When \(d=2\),
		\begin{align*}
			\left|\frac{1}{|E|}\int_E m_v-v_h|_T\mathrm{d}s\right|^2
			\lesssim &\sum_{F\in \mathcal F_h, E\subset F}
			\left|\jump{v_h(E)}\right|^2
			\lesssim 
			\sum_{F\in \mathcal F_h, E\subset F}
			h^{-1}_F
			\|\jump{v_h}\|^2_{0,F}.
		\end{align*}
		Hence, by shape regularity and the inverse estimate on each face,
		\begin{equation*}
			\sum_{T\in\mathcal T_h}
			\sum_{E\in\mathcal E_h\cap \partial T}
			h_T^{d-4}
			\left|\frac{1}{|E|}\int_E m_v-v_h|_T\mathrm{d}s\right|^2
			\lesssim
			\sum_{F\in\mathcal F_h}
			h_F^{-3}\|\jump{v_h}\|_{0,F}^2.
		\end{equation*}
		For the face degree of freedom, the inverse trace inequality and local projection estimate imply
		\begin{equation*}
			h^{d-2}_T\left|
			\frac{1}{|F|}\int_F (\psi_h-\nabla_hv_h|_T)\cdot n_F\mathrm{d}s\right|^2
			\lesssim h^{-2}_T\|\nabla_h v_h-Q_0\psi_h\|^2_{0,T}+\|\nabla_h\psi_h\|^2_{0,T}.
		\end{equation*}
		Combining this estimate Lemma \ref{lem:bih-strong-jump} yields
		\begin{equation}
			\sum_{T\in\mathcal T_h}
			h_T^{-4}
			\|m_v-v_h\|_{0,T}^2
			\lesssim
			\iota^{-2}\tnorm{\bm v_h}^2.
			\label{eq:bih-morley-difference-bound}
		\end{equation}
		
		Since $\bm v_h$ is piecewise‑linear, it follows from the inverse inequality that
		\[
		\|\nabla_h^2m_v\|_0^2
		=\|\nabla_h^2(m_v-v_h)\|_0^2
		\lesssim\sum_{T\in\mathcal T_h}
		h_T^{-4}
		\|m_v-v_h\|_{0,T}^2.
		\]
		Moreover,
		\[
		\nabla_h m_v-\psi_h
		=
		\nabla_h(m_v-v_h)
		+
		(\nabla_h v_h-Q_0\psi_h)
		+
		(Q_0\psi_h-\psi_h),
		\]
		so that
		\[
		\sum_{T\in\mathcal T_h}
		h_T^{-2}
		\|\nabla_h m_v-\psi_h\|_{0,T}^2
		\lesssim
		\iota^{-2}\tnorm{\bm v_h}^2.
		\]
	\end{proof}
	
	\begin{remark}
		For \(d=2,3\), Wang et al.~\cite{wang2018morley}
		constructed an enriching operator
		\[
		E_h:M_{h0}
		\rightarrow
		A_{h0}\subset H_0^2(\Omega),
		\]
		where \(A_{h0}\subset H_0^2(\Omega)\) denotes the conforming
		finite element space built from the quintic Argyris triangle
		\cite{argyris1968tuba} when \(d=2\), and from the degree-ten member
		of Zhang's family of \(C^1\) tetrahedral elements
		\cite{zhang2009family} when \(d=3\). For these choices of \(A_{h0}\), the degree-of-freedom averaging
		construction of the enriching operator and the corresponding local
		estimates are given in
		\cite{wang2018morley}.
		Together with the standard inverse inequalities, these estimates imply
		\begin{equation}
			\label{eq:bih-morley-companion-bound}
			\sum_{T\in\Tcal_h}\sum_{j=0}^2
			h_T^{2j-4}
			\norm{\nabla_h^j(m_h-E_hm_h)}_{0,T}^2
			\lesssim
			\norm{\nabla_h^2m_h}_0^2,\qquad \forall m_h\in M_{h0}.
		\end{equation}
		\label{rem:bih-morley-companion-approx}
		
		The conforming companion \(E_h m_v\) provides a direct approximation
		of the original discrete pair \((v_h,\psi_h)\).
		Indeed, by the triangle inequality together with
		Lemma~\ref{lem:bih-morley-companion} and
		\eqref{eq:bih-morley-companion-bound},
		\begin{align}
			\sum_{T\in\Tcal_h}\Big(
			&h_T^{-4}\norm{v_h-E_h m_v}_{0,T}^2
			+h_T^{-2}\norm{\psi_h-\nabla E_h m_v}_{0,T}^2
			+\norm{\nabla^2E_h m_v}_{0,T}^2
			\Big)
			\lesssim
			\iota^{-2}\tnorm{\bm v_h}^2.
			\label{eq:bih-morley-companion-direct-approx}
		\end{align}
	\end{remark}
	
	To get rid of the factor $\iota^{-2}$ in the estimate, we introduce an average interpolation operator \(\Pi_{\rm A}:\,\rm{CR}_{h0}\to A_{h0}\). For each \(v_h\in\mathrm{CR}_{h0}\), define
\(\Pi_{\rm A}v_h\in A_{h0}\) by
\begin{equation}
    \label{eq:cr-everage-operator}
    \mathcal N(\Pi_{\rm A}v_h)
    =\frac{1}{\#\omega_{\mathcal N}}
    \sum_{T\in\omega_{\mathcal N}}\mathcal N(v_h|_T)
\end{equation}
for every free degree of freedom \(\mathcal N\) of \(A_{h0}\),
where \(\omega_{\mathcal N}\) denotes the set of elements
sharing \(\mathcal N\). We use the same notation \(\Pi_{\rm A}\) for its componentwise extension to vector-valued functions. By the standard averaging result for the CR element
	\cite[Lemma~5.6.4]{ShiWang2010}, we have 
	\begin{equation}
		\label{eq:H1-consistency-estimate}
		\left(\sum_{T\in \mathcal T_h}
		h^{-2}_T\norm{v_h-\Pi_{\rm A}v_h}^2_{0,T}\right)^{1/2}
		+\norm{\nabla_h(v_h-\Pi_{\rm A}v_h)}_0
		\lesssim
		\norm{\nabla_h v_h}_0.
	\end{equation}
	
	\begin{lemma}
		\label{lem:bih-consistency-big-e}
		Let \(u\in H^{2+s}(\Omega)\cap H_0^2(\Omega)\) solve
		\eqref{eq:bih-weak}, \(0<s\le 1\). Then
		\begin{equation}
			\label{eq:bih-consistency-estimate-big-e}
            \sup_{0\ne \bm v_h\in V^{\rm sp}_{h0}}
            \frac{\left|a_h^{\rm sp}(\Pi_h^{\rm sp}u,\bm v_h)-(f,v_h)\right|}{\tnorm{\bm v_h}}
			\lesssim
			\iota h^s|u|_{2+s}
			+\iota^{-1} h^2|u|_2
			+\iota^{-1}h^2\norm{f}_0.
		\end{equation}
	\end{lemma}
	
	\begin{proof}
		Let \(m_v\) be given by
		Lemma~\ref{lem:bih-morley-companion}. Since
		\(E_h m_v\in H_0^2(\Omega)\), the weak formulation gives
		\[
		\iota^2(\nabla^2u,\nabla^2E_h m_v)
		+(\nabla u,\nabla E_h m_v)
		=(f,E_h m_v).
		\]
		Hence, by the commuting property \eqref{eq:cr-commuting-general} and \eqref{eq:bih-weak},
		\begin{align}
			a_h^{\rm sp}(\Pi_h^{\rm sp}u,\bm v_h)-(f,v_h)
			={}&
			\iota^2
			(\nabla^2u,\nabla_h\psi_h-\nabla^2E_h m_v)
			+(\nabla u,\nabla_h(v_h-E_h m_v))
			\notag\\
			&
			-(f,v_h-E_h m_v).
			\label{eq:bih-morley-consistency-split}
		\end{align}
		
		For the first term in \eqref{eq:bih-morley-consistency-split}, inserting the elementwise
		\(L^2\)-projection gives
		\begin{align}
			(\nabla^2u,\nabla_h\psi_h-\nabla^2E_h m_v)
			={}&
			(\nabla^2u-Q_0 \nabla^2u,
			\nabla_h\psi_h-\nabla^2E_h m_v)
			\notag\\
			&+
			(Q_0 \nabla^2u,
			\nabla_h\psi_h-\nabla^2E_h m_v).
			\label{eq:bih-consistency-Q0-split}
		\end{align}
		By Lemma~\ref{lem:bih-fractional-means} and
		\eqref{eq:bih-morley-companion-direct-approx},
		\begin{align}
			\left|
			(\nabla^2u-Q_0 \nabla^2u,
			\nabla_h\psi_h-\nabla^2E_h m_v)
			\right|
			\lesssim
			h^s|u|_{2+s}
			\iota^{-1}\tnorm{\bm v_h}.
			\label{eq:bih-consistency-Q0-error}
		\end{align}
		For the second term, since \(Q_0 \nabla^2u\) is elementwise
		constant, elementwise integration by parts gives
		\begin{align*}
			(Q_0 \nabla^2u,
			\nabla_h\psi_h-\nabla^2E_h m_v)
			=
			\sum_{F\in\Fcal_h^i}
			\left(
			\jump{Q_0 \nabla^2u}n_F,
			\{\psi_h-\nabla E_h m_v\}_F
			\right)_F.
		\end{align*}
		Indeed, the term involving the average of \(Q_0 \nabla^2u\)
		vanishes because
		\(\int_F\jump{\psi_h-\nabla E_h m_v}\,ds=0
		\).
		By Cauchy's inequality and Lemma~\ref{lem:bih-fractional-means},
		\begin{align*}
			\left|
			(Q_0 \nabla^2u,
			\nabla_h\psi_h-\nabla^2E_h m_v)
			\right|\lesssim
			h^s\seminorm{u}_{2+s}
			\left(
			\sum_{F\in\Fcal_h^i}
			h_F^{-1}
			\norm{\{\psi_h-\nabla E_h m_v\}_F}_{0,F}^2
			\right)^{1/2}.
		\end{align*}
		The trace inequality and
		\eqref{eq:bih-morley-companion-direct-approx} further imply
		\[
		\left(
		\sum_{F\in\Fcal_h^i}
		h_F^{-1}
		\norm{\{\psi_h-\nabla E_h m_v\}_F}_{0,F}^2
		\right)^{1/2}
		\lesssim
		\iota^{-1}\tnorm{\bm v_h}.
		\]
		Hence
		\begin{equation}
			\left|
			(Q_0 \nabla^2u,
			\nabla_h\psi_h-\nabla^2E_h m_v)
			\right|
			\lesssim
			\iota^{-1}h^s\seminorm{u}_{2+s}
			\tnorm{\bm v_h}.
		\end{equation}
		
		For the second term in \eqref{eq:bih-morley-consistency-split}, applying \eqref{eq:bih-morley-companion-direct-approx} and arguing as in the proof of \eqref{eq:bih-consistency-Q0-split}, we obtain
		\begin{equation}
			\label{eq:bih-consistency-lower-gradient}
			|(\nabla u,\nabla_h(v_h-E_h m_v))|
			\lesssim 
			\iota^{-1} h^2|u|_2\tnorm{\bm v_h},
		\end{equation}
		moreover
		\begin{equation}
			\label{eq:bih-consistency-lower-load}
			|(f,v_h-E_hm_v)|\lesssim \iota^{-1}h^2 \|f\|_0\tnorm{\bm v_h}.
		\end{equation}
		Finally, combining \eqref{eq:bih-morley-consistency-split}-\eqref{eq:bih-consistency-lower-load}, we obtain \eqref{eq:bih-consistency-estimate-big-e}.
	\end{proof}
	
	\begin{lemma}
		\label{lem:bih-consistency-small-e}
		Let \(u\in H_0^2(\Omega;\mathbb R^d)\) solve \eqref{eq:bih-weak} and \(u_0\in H^{2}(\Omega;\mathbb R^d)\cap H^1_0(\Omega;\mathbb R^d)\) solve \eqref{eq:Poisson-model}. Then
		\begin{equation}
			\label{eq:bih-consistency-estimate-small-e}
            \sup_{0\ne \bm v_h\in V^{\rm sp}_{h0}}\frac{\left|a_h^{\rm sp}(\Pi_h^{\rm sp}u,\bm v_h)-(f,v_h)\right|}{\tnorm{\bm v_h}}
			\lesssim
			\iota|u|_2
			+|u-u_0|_1
			+h|u_0|_{2}
			+h\norm{f}_0.
		\end{equation}
	\end{lemma}
	
	\begin{proof}
		Inserting \(u_0\) gives the following equation:
		\begin{align*}
			a_h^{\rm sp}(\Pi_h^{\rm sp}u,\bm v_h)-(f,v_h)
			={}&
			\iota^2(\nabla^2u,\nabla_h\psi_h)
			+(\nabla(u-u_0),\nabla_h v_h)\\
			&+(\nabla u_0,\nabla_h v_h)-(f,v_h).
		\end{align*}
		
		Since \(\Pi_{\rm A}v_h\in H_0^1(\Omega;\mathbb R^d)\), the weak formulation of \eqref{eq:Poisson-model} yields
		\[(\nabla u_0,\nabla_h v_h)-(f,v_h)=\bigl(\nabla u_0,\nabla_h (v_h-\Pi_{\rm A}v_h)\bigr)-(f,v_h-\Pi_{\rm A}v_h)\]
		Inserting \(Q_0\nabla u_0\) into the first term and arguing as in
		\eqref{eq:bih-consistency-Q0-split}, Lemma~\ref{lem:bih-fractional-means} and \eqref{eq:H1-consistency-estimate} 
		give
		\[
		\left|
		(\nabla u_0,\nabla_h(v_h-\Pi_{\rm A}v_h))
		\right|
		\lesssim
		h \seminorm{u_0}_{2}\norm{\nabla_h v_h}_0.
		\]
		Moreover,
		\[
		\left|(f,v_h-\Pi_{\rm A}v_h)\right|
		\lesssim
		h\norm{f}_0\norm{\nabla_h v_h}_0.
		\]
		Combining with \eqref{eq:bih-H1-control}, we obtain
		\begin{equation}
			\label{eq:bih-consistency-lower-load-u0}
			\left|
			(\nabla u_0,\nabla_h v_h)-(f,v_h)
			\right|
			\lesssim
			\left(
			h\seminorm{u_0}_{2}
			+h\norm{f}_0
			\right)
			\tnorm{\bm v_h}.
		\end{equation}
		and then we have \eqref{eq:bih-consistency-estimate-small-e}.
	\end{proof}

    \begin{theorem}
        Let \(0<s\le 1\). Let \(u\in H^{2+s}(\Omega;\mathbb R^d)\cap H^2_0(\Omega;\mathbb R^d)\) and \(\bm u_h\in V^{\rm sp}_{h0}\) be the solutions of \eqref{eq:bih-weak} and \eqref{eq:bih-discrete}, respectively. Under Assumption~\ref{assum:Regularity}, we have 
        \begin{equation}
            |(u,\nabla u)-\bm u_h|_{\rm E}
		\lesssim
		\min\{\iota h^s|u|_{2+s}
            +h|u|_2
			+\iota^{-1} h^2|u|_2
			+\iota^{-1}h^2\norm{f}_0,(\iota^{1/2}+h)\norm{f}_0\}.
        \end{equation}
    \end{theorem}
    
    \subsection{Equivalence with the RT--CR method}
In this section, we will introduce the equivalence between method
\eqref{eq:bih-discrete} and the mixed method proposed in
\cite{HuangTang2025}. This fact also implies the intrinsic selection
of the stabilized parameter $h_T^{-2}$.
For simplicity, we will omit the second-order term in \eqref{eq:bih-discrete}. For each $T\in\mathcal T_h$, define
\[
x_T=\frac{1}{|T|}\int_T x\,dx,
\quad
\gamma_T=\frac{d^2|T|}{\int_T|x-x_T|^2\,dx}.
\]
We use the local matrix space
\[
\Sigma(T)
=
\mathbb P_0(T;\mathbb R^{d\times d})
\oplus
\bigl\{
a\otimes(x-x_T):a\in\mathbb R^d
\bigr\},
\]
whose rows belong to the lowest-order Raviart--Thomas space,
and define
\[
\Sigma_h
=
\bigl\{
\tau_h\in H(\operatorname{div},\Omega;\mathbb R^{d\times d})
:
\tau_h|_T\in\Sigma(T)
\text{ for all }T\in\mathcal T_h
\bigr\}.
\]
Here the divergence acts row-wise, and
$(\nabla\phi)_{ij}=\partial_j\phi_i$. The RT--CR method seeks
$(\sigma_h,u_h)\in\Sigma_h\times V_h$ such that
\begin{subequations}\label{eq:rtcr}
\begin{align}
(\sigma_h,\tau_h)
+(\operatorname{div}\tau_h,\nabla_hu_h)
&=0
&&\forall\tau_h\in\Sigma_h,
\label{eq:rtcr-first}\\
-(\operatorname{div}\sigma_h,\nabla_hv_h)
&=(f,v_h)
&&\forall v_h\in V_h.
\label{eq:rtcr-second}
\end{align}
\end{subequations}
The corresponding CR--CR method seeks
$(u_h,\phi_h)\in V_h\times W_h$ such that
\begin{equation}\label{eq:crcr-weighted}
\begin{aligned}
&(\nabla_h\phi_h,\nabla_h\psi_h)
+\sum_{T\in\mathcal T_h}\gamma_T
(\nabla_hu_h-Q_0\phi_h,
 \nabla_hv_h-Q_0\psi_h)_T
=(f,v_h)
\end{aligned}
\end{equation}
for all $(v_h,\psi_h)\in V_h\times W_h$.

\begin{lemma}\label{prop:rtcr-crcr-equivalence}
Given a solution $(u_h,\phi_h)$ of
\eqref{eq:crcr-weighted}, the reconstruction
\begin{equation}\label{eq:rtcr-stress-reconstruction}
\sigma_h|_T
=
\nabla\phi_h|_T
+\frac{\gamma_T}{d}
\bigl(Q_{0,T}\phi_h-\nabla u_h|_T\bigr)
\otimes(x-x_T)
\end{equation}
belongs to $\Sigma_h$ and together with $u_h$, solves
\eqref{eq:rtcr}.
Conversely, given a solution $(\sigma_h,u_h)$ of
\eqref{eq:rtcr}, the function defined by
\begin{equation}\label{eq:rtcr-cr-reconstruction}
\phi_h|_T(x)
=
\nabla u_h|_T
+\gamma_T^{-1}\operatorname{div}\sigma_h|_T
+(Q_{0,T}\sigma_h)(x-x_T)
\end{equation}
belongs to $W_h$ and, together with $u_h$, solves
\eqref{eq:crcr-weighted}.
These two reconstructions are mutually inverse.
\end{lemma}

\begin{proof}
For a scalar affine function $v$ and a vector-valued affine
function $\psi$ on $T$, define
\[
H_T(v,\psi)
=
\nabla\psi
+\frac{\gamma_T}{d}
(Q_{0,T}\psi-\nabla v)\otimes(x-x_T).
\]
Since $\int_T(x-x_T)\,dx=0$, the constant and linear parts
are $L^2(T)$-orthogonal. The definition of $\gamma_T$
therefore gives
\begin{equation}\label{eq:rtcr-local-energy}
\begin{aligned}
(H_T(u,\phi),H_T(v,\psi))_T
={}&(\nabla\phi,\nabla\psi)_T+\gamma_T
(Q_{0,T}\phi-\nabla u,
 Q_{0,T}\psi-\nabla v)_T.
\end{aligned}
\end{equation}
Moreover, direct calculation and integration by parts yield
\begin{equation}\label{eq:rtcr-local-identity}
\begin{aligned}
(H_T(v,\psi),\tau)_T
&=(\nabla\psi,\tau)_T
  +(Q_{0,T}\psi-\nabla v,\operatorname{div}\tau)_T\\
&=-(\nabla v,\operatorname{div}\tau)_T
  +\sum_{F\subset\partial T}
   \langle Q_F\psi,\tau n_T\rangle_F
\end{aligned}
\end{equation}
for every $\tau\in\Sigma(T)$.
Here we used that $\operatorname{div}\tau$ is constant on $T$
and $\tau n_T$ is constant on each face.

Suppose first that $(u_h,\phi_h)$ solves
\eqref{eq:crcr-weighted}, and set
$\sigma_h|_T=H_T(u_h,\phi_h)$.
By \eqref{eq:rtcr-local-energy} and
\eqref{eq:rtcr-local-identity},
\begin{equation}\label{eq:rtcr-reconstructed-equation}
-\sum_{T\in\mathcal T_h}
(\operatorname{div}\sigma_h,\nabla v_h)_T
+\sum_{T\in\mathcal T_h}
 \sum_{F\subset\partial T}
 \langle Q_F\psi_h,\sigma_hn_T\rangle_F
=(f,v_h).
\end{equation}
Taking $v_h=0$ and choosing the interior face averages of
$\psi_h$ independently shows that
\[
\sigma_h|_{T^+}n_{T^+}
+\sigma_h|_{T^-}n_{T^-}=0
\qquad
\text{on every interior face }
F=\partial T^+\cap\partial T^-.
\]
Hence $\sigma_h\in\Sigma_h$.
Taking $\psi_h=0$ in
\eqref{eq:rtcr-reconstructed-equation} gives
\eqref{eq:rtcr-second}.
For any $\tau_h\in\Sigma_h$, summing
\eqref{eq:rtcr-local-identity} over the mesh gives
\[
(\sigma_h,\tau_h)
=-(\nabla_hu_h,\operatorname{div}\tau_h),
\]
because the interior face contributions cancel and the
boundary face averages of $\phi_h$ vanish.
Thus \eqref{eq:rtcr-first} also holds.

Conversely, let $(\sigma_h,u_h)$ solve \eqref{eq:rtcr},
and define the piecewise affine function $\phi_h$ by
\eqref{eq:rtcr-cr-reconstruction}.
Since $\sigma_h|_T\in\Sigma(T)$, it admits the representation
\[
\sigma_h|_T
=
Q_{0,T}\sigma_h
+\frac{1}{d}
(\operatorname{div}\sigma_h|_T)\otimes(x-x_T).
\]
Consequently,
\[
\nabla\phi_h|_T=Q_{0,T}\sigma_h,
\qquad
Q_{0,T}\phi_h-\nabla u_h|_T
=\gamma_T^{-1}\operatorname{div}\sigma_h|_T,
\]
and hence $\sigma_h|_T=H_T(u_h,\phi_h)$.
Using \eqref{eq:rtcr-local-identity} and
\eqref{eq:rtcr-first}, we obtain
\[
\sum_{T\in\mathcal T_h}
\sum_{F\subset\partial T}
\langle Q_F\phi_h,\tau_hn_T\rangle_F
=0
\qquad\forall\tau_h\in\Sigma_h.
\]
The constant normal traces of $\tau_h\in\Sigma_h$ can be
prescribed independently on all mesh faces.
It follows that the averages of $\phi_h$ agree across
interior faces and vanish on boundary faces.
Therefore $\phi_h\in W_h$.

Finally, for any $(v_h,\psi_h)\in V_h\times W_h$,
the normal continuity of $\sigma_h$ and
\eqref{eq:rtcr-local-identity} imply
\[
\begin{aligned}
\sum_{T\in\mathcal T_h}
(H_T(u_h,\phi_h),H_T(v_h,\psi_h))_T
&=-(\operatorname{div}\sigma_h,\nabla_hv_h)\\
&=(f,v_h).
\end{aligned}
\]
Together with \eqref{eq:rtcr-local-energy}, this proves
\eqref{eq:crcr-weighted}.
The reconstruction \eqref{eq:rtcr-cr-reconstruction}
recovers both the elementwise mean and the gradient of
$\phi_h$ in \eqref{eq:rtcr-stress-reconstruction}.
Thus the two reconstructions are mutually inverse.
\end{proof}

\begin{remark}
The exact equivalence requires
$\gamma_T=d^2|T|/J_T$.
Since $\gamma_T\simeq h_T^{-2}$ on a shape-regular mesh,
replacing $\gamma_T$ by $h_T^{-2}$ gives a uniformly
equivalent bilinear form, but generally does not preserve
the exact correspondence between the discrete solutions.
\end{remark}
	\section{The strain-gradient elasticity problem}
	\label{sec:sge}
	Problem \eqref{eq:sge-model} admits the following variational formulation: find \(u\in H^2_0(\Omega;\mathbb R^d)\) such that 
	\begin{equation}
		\label{eq:sge-weak}
		\iota^2(\nabla\sigma(u),\nabla\varepsilon(v))+(\sigma(u),\varepsilon(v))=(f,v),\quad\forall v\in H^2_0(\Omega;\mathbb R^d).
	\end{equation}
	The strain gradient \(\nabla \varepsilon(v)\) and the stress gradient \(\nabla \sigma(v)\) are third‑order tensors defined respectively by \((\nabla\varepsilon(v))_{ijk}=\partial \varepsilon_{ij}/\partial x_k\) and \((\nabla\sigma(v))_{ijk}=\partial \sigma_{ij}/\partial x_k\).
	
	The well-posedness of \eqref{eq:sge-weak} follows from the \(H^2\)-Korn inequality~\cite{li2017two,li2021h2korn}:
    \begin{equation}
        \label{eq:H2-korn}
        \|\varepsilon(v)\|_0^2+\|\nabla\varepsilon(v)\|_0^2
        \ge C\|\nabla v\|_1^2
        \qquad\text{for all }v\in H^2_0(\Omega;\mathbb R^d),
    \end{equation}
    where \(C=1/2\) for \(d=2\) and \(C=1-1/\sqrt{2}\) for \(d=3\).
	
	Then, we introduce a reduced formulation and state the regularity assumptions. Taking \(\iota=0\), problem \eqref{eq:sge-model} becomes the linear elasticity model. 
	\begin{equation}
		\label{eq:sge-LE}
		-\operatorname{div}\sigma(u_0)=f
		\quad\text{in }\Omega,
		\qquad
		u_0=0
		\quad\text{on }\partial\Omega.
	\end{equation}
	\begin{assumption}
		\label{ass:le-regularity}
		Assume that \(f\in L^2(\Omega;\mathbb R^d)\), let \(u_0\in H^1_0(\Omega;\mathbb R^d)\) solve \eqref{eq:sge-LE}, then
		\begin{equation}
			\label{eq:sge-LE-regularity}
			\|u_0\|_2+\lambda\|\mathrm{div}\,u_0\|_1\lesssim\|f\|_0.
		\end{equation}
		The hidden constant is independent of $\lambda$.
	\end{assumption}
	\begin{assumption}
		\label{ass:sge-regularity}
		Assume that $f\in L^2(\Omega;\mathbb R^d)$, let $u$ be the solution of
		\eqref{eq:sge-weak}, and $u_0\in H^1_0(\Omega;\mathbb R^d)$ solve \eqref{eq:sge-LE},
		then 
		\begin{align}
			\label{eq:sge-regularity-div}
			\lambda\|\operatorname{div}(u-u_0)\|_0
			+\lambda\iota\|\operatorname{div}u\|_1
			+\lambda\iota^2\|\operatorname{div}u\|_2
			\lesssim
			\iota^{1/2}\|f\|_0,\\
			\label{eq:sge-regularity-nabla}
			|u-u_0|_1+\iota|u|_2+\iota^2|u|_3
			\lesssim
			\iota^{1/2}\|f\|_0.
		\end{align}
		The hidden constants are independent of $\iota$ and $\lambda$.
	\end{assumption}
	\begin{remark}
		In fact, Assumption \ref{ass:le-regularity} holds when \(\Omega\) is a convex polyhedral domain \cite{brenner1992linear,maz_i_a2010elliptic}, whereas the currently available proof of Assumption \ref{ass:sge-regularity} requires the boundary \(\partial\Omega\) to be smooth \cite{liao2023taylor,chen2023robust}. In particular, for a bounded polygonal domain \(\Omega\), the regularity estimate \eqref{eq:sge-regularity-nabla} holds \cite{li2018new}, the hidden constants are independent of  \(\iota\).
	\end{remark}
	
	For \(v\in H^2(\Omega;\mathbb{R}^d)\), we have
	\(\varepsilon(v)\in H^1(\Omega;\mathbb{S})\), where \(\mathbb{S}\) denotes the space of symmetric matrices.  We construct a nonconforming discretization of the product space
	\(H^2(\Omega;\mathbb{R}^d)\times H^1(\Omega;\mathbb{S})\): For any \(T\in \mathcal T_h\), the local shape function space is \((\mathbb P_1(T)\otimes\mathbb R^d)\times(\mathbb P_1(T)\otimes\mathbb S)\) and the nodal variables are 
	\begin{equation}
		\label{eq:vh-dofs}
		\left\{\frac{1}{|F|}\int_F v\mathrm{d}s,\quad\frac{1}{|F|}\int_F\psi\mathrm{d}s,\qquad\forall F\in \mathcal F_h,~F\subset \partial T\right\}, 
	\end{equation}
	for \((v,\psi)\in (\mathbb P_1(T)\otimes\mathbb R^d)\times(\mathbb P_1(T)\otimes\mathbb S)\). The global space
	\begin{equation}
		\label{eq:vh-global}
		V_h^{\rm sg}:=\{(v,\psi)\in (\mathbb P_1(\mathcal T_h)\otimes\mathbb R^d)\times(\mathbb P_1(\mathcal T_h)\otimes\mathbb S)
		:\text{\eqref{eq:vh-dofs} are single-valued } \forall F\in \mathcal F^i_h\}.
	\end{equation}
	The homogeneous subspace is
	\begin{equation}
		\label{eq:sge-space-zero}
		V_{h0}^{\rm sg}
		:=
		\left\{(v_h,\psi_h)\in V_h^{\rm sg}:
		\int_F v_h\,\mathrm{d}s=0,
		\ \int_F\psi_h\,\mathrm{d}s=0,
		\ \forall F\in\mathcal F^\partial_h
		\right\}.
	\end{equation}
	For $v\in H^2(\Omega;\mathbb R^d)$, define
	\begin{equation}
		\label{eq:sge-Qh}
		\Pi_h^{\rm sg}v:=(\Pi_{\rm CR} v,\Pi_{\rm CR}\varepsilon(v)).
	\end{equation}
	If $v\in H^2_0(\Omega;\mathbb R^d)$, then
	$\Pi_h^{\rm sg}v\in V_{h0}^{\rm sg}$.
	
	The discrete problem is: find
	$\bm u_h=(u_h,\phi_h)\in V_{h0}^{\rm sg}$ such that
	\begin{equation}
		\label{eq:sge-discrete}
		a_h^{\rm sg}(\bm u_h,\bm v_h)
		+s_h^{\rm sg}(\bm u_h,\bm v_h)
		=(f,v_h)
		\qquad
		\forall\bm v_h=(v_h,\psi_h)\in V_{h0}^{\rm sg}.
	\end{equation}
	where
	\begin{align}
		a_h^{\rm sg}(\bm u_h,\bm v_h)
		:=&\ 2\mu\iota^2(\nabla_h\phi_h,\nabla_h\psi_h)
		+\lambda\iota^2(\nabla_h\operatorname{tr}\phi_h,
		\nabla_h\operatorname{tr}\psi_h)\notag\\
		&+2\mu(\varepsilon_h(u_h),\varepsilon_h(v_h))
		+\lambda(\operatorname{div}_h u_h,\operatorname{div}_h v_h),
		\label{eq:sge-ah}\\
		s_h^{\rm sg}(\bm u_h,\bm v_h)
		:=&\ \sum_{F\in\mathcal F^i_h}h_F^{-1}
		\langle\jump{u_h},\jump{v_h}\rangle_F\notag\\
		&+\iota^2\sum_{T\in\mathcal T_h}h_T^{-2}
		(\varepsilon_h(u_h)-Q_0\phi_h,
		\varepsilon_h(v_h)-Q_0\psi_h)_T .
		\label{eq:sge-s}
	\end{align}
	Associated with the bilinear form \(a_h^{\rm sg}+s_h^{\rm sg}\), we define a semi-norm \(\norme{\cdot}\) by
	\begin{equation}
		\label{eq:norme}
		\norme{\bm v}^2=a_h^{\rm sg}(\bm v,\bm v)+s_h^{\rm sg}(\bm v, \bm v),\qquad \forall \bm v:=(v,\psi)\in \bigl(H^2_0(\Omega;\mathbb R^d)\times H^1_0(\Omega;\mathbb S)\bigr)+V_h^{\rm sg}.
	\end{equation}
	
	To establish the well-posedness of the above discrete scheme, we introduce the following discrete Korn inequality from \cite[Theorem 3.1]{brenner2004korn}.
	\begin{lemma}
		\label{lem:sge-discrete-korn}
		For every
		$v_h\in\mathrm{CR}_{h,0}\otimes\mathbb R^d$,
		\begin{equation}
			\label{eq:sge-discrete-korn}
			\|v_h\|_0^2+\|\nabla_h v_h\|_0^2
			\lesssim
			\|\varepsilon_h(v_h)\|_0^2
			+\sum_{F\in\mathcal F^i_h}h_F^{-1}
			\|\jump{v_h}\|_{0,F}^2.
		\end{equation}
	\end{lemma}

    For \(\bm v_h\in V_{h0}^{\rm sg}\), assume \(\norme{\bm v_h}=0\). By Lemma~\ref{lem:sge-discrete-korn}, $v_h=0$, and hence
		$Q_0\psi_h=0$. Since $\nabla_h\psi_h=0$, the tensor field $\psi_h$
		is elementwise constant, so that $\psi_h=Q_0\psi_h=0$. Hence, \(\norme{\cdot}\) defines a norm on \(V_{h0}^{\rm sg}\).
	
	
	For \(\bm v\in H^2_0(\Omega;\mathbb R^d)\times H^1_0(\Omega;\mathbb S)+V_h^{\rm sg}\), we define the energy seminorm
	\[ |\bm v|_{\mathrm E}^2:=a_h^{\rm sg}(\bm v,\bm v),\] 
	and establish the corresponding error estimates for the CR-CR method \eqref{eq:sge-discrete} with respect to this seminorm. Let \(u\) be the solution of \eqref{eq:sge-weak}. Let \(\bm u_h =(u_h,\phi_h)\) be the solution of \eqref{eq:sge-discrete}. By triangle inequality and the definition of \(\norme{\cdot}\), we first obtain
	\[|(u,\varepsilon(u))-\bm u_h|_{\rm E}\le |(u,\varepsilon(u))-\Pi_h^{\rm sg} u|_{\rm E}
	+\norme{\Pi_h^{\rm sg}u-\bm u_h}.\]

    \begin{lemma}
    \label{lem:sge-interpolation-error}
        Let \(u\in H^{2+s}(\Omega; \mathbb R^d)\cap H^2_0(\Omega;\mathbb R^d)\), \(0<s\le 1\) be the solution of \eqref{eq:sge-weak}, and \(u\) satisfy \eqref{eq:sge-regularity-div} and \eqref{eq:sge-regularity-nabla}, then
        \begin{equation}
        \label{eq:sge-interpolation-error}
            |(u,\varepsilon(u))-\Pi_h^{\rm sg} u|_{\rm E}^2
            \lesssim \min\{h^s\iota(|u|_{2+s}+\sqrt{\lambda}|\operatorname{div}u|_{1+s})+h(|u|_2+\sqrt{\lambda}|\operatorname{div}u|_1),(\iota^{1/2}+h)\norm{f}_0\}
        \end{equation}  
    \end{lemma}

    \begin{proof}
        By Lemma~\ref{lem:cr-commuting-general}, we have
	\begin{align}
		|(u,\varepsilon(u))-\Pi_h^{\rm sg} u|_{\rm E}^2
		={}&
		2\mu\iota^2
		\|(I-Q_0)\nabla\varepsilon(u)\|_0^2
		+
		\lambda\iota^2
		\|(I-Q_0)\nabla\operatorname{div}u\|_0^2
		\notag\\
		&+
		2\mu
		\|(I-Q_0)\varepsilon(u)\|_0^2
		+
		\lambda
		\|(I-Q_0)\operatorname{div}u\|_0^2,
	\end{align}
	where \(I\) denotes the identity operator. Then
	\begin{align}
		\label{eq:interpolation-estimate1}
		|(u,\varepsilon(u)) -\Pi_h^{\rm sg} u|_{\rm E}
		\lesssim &\iota\|(I-Q_0)\nabla^2 u\|_0
		+\iota\sqrt{\lambda}\|(I-Q_0)\nabla\operatorname{div}u\|_0\notag\\
		&+\|(I-Q_0)\nabla u\|_0
		+\sqrt{\lambda}\|(I-Q_0)\operatorname{div} u\|_0\notag\\
        \lesssim & h^s\iota(|u|_{2+s}+\sqrt{\lambda}|\operatorname{div}u|_{1+s})+h(|u|_2+\sqrt{\lambda}|\operatorname{div}u|_1).
	\end{align}
	Let \(u_0\) be the solution of \eqref{eq:sge-LE}, by \(u=(u-u_0)+u_0\), we obtain 
	\begin{align}
		\label{eq:interpolation-estimate2}
		|(u,\varepsilon(u)) -\Pi_h^{\rm sg} u|_{\rm E}
		\lesssim &\iota|u|_2
		+\iota\sqrt{\lambda}|\operatorname{div}u|_1\notag\\
		&+|u-u_0|_1
		+\sqrt{\lambda}\|\operatorname{div} (u-u_0)\|_0\notag\\
		& +\|(I-Q_0)\nabla u_0\|_0
		+\sqrt{\lambda}\|(I-Q_0)\operatorname{div} u_0\|_0\notag\\
        \lesssim & (\iota^{1/2}+h)\norm{f}_0.
	\end{align}
    \end{proof}
	
	By the discrete problem \eqref{eq:sge-discrete}, we have
	\begin{align*}
		\norme{\Pi_h^{\rm sg} u-\bm u_h}^2
		={}&
		a_h^{\rm sg}(\Pi_h^{\rm sg} u,\Pi_h^{\rm sg} u-\bm u_h)
		+
		s_h^{\rm sg}(\Pi_h^{\rm sg} u,\Pi_h^{\rm sg} u-\bm u_h)
		\notag\\
		&-
		(f,\Pi_{\rm CR}u-u_h).
	\end{align*}
	Therefore,
	\[
	\norme{\Pi_h^{\rm sg} u-\bm u_h}
	\le
	\sup_{0\ne\bm v_h\in V_{h0}^{\rm sg}}
	\frac{
		a_h^{\rm sg}(\Pi_h^{\rm sg} u,\bm v_h)
		+
		s_h^{\rm sg}(\Pi_h^{\rm sg} u,\bm v_h)
		-
		(f,v_h)
	}{
		\norme{\bm v_h}
	}.
	\]

    \begin{lemma}
    \label{lem:stabilization-error}
        Let \(u\in H^{2+s}(\Omega; \mathbb R^d)\cap H^2_0(\Omega;\mathbb R^d)\), \(0<s\le 1\) be the solution of \eqref{eq:sge-weak}, and \(u\) satisfy \eqref{eq:sge-LE-regularity}, \eqref{eq:sge-regularity-div} and \eqref{eq:sge-regularity-nabla}, then 
        \begin{equation}
        \label{eq:stabilization-error}
            \sup_{0\ne \bm v_h \in V^{\rm sg}_{h0}}\frac{s^{\rm sg}_h(\Pi^{\rm sg}_hu,\bm v_h)}{\norme{\bm v_h}}\lesssim \min \{h^s\iota|u|_{2+s}+h|u|_2,(\iota^{1/2}+h)\norm{f}_0\}.
        \end{equation}
    \end{lemma}

    \begin{proof}
        Since \(s_h^{\rm sg}\) is a symmetric positive semi-definite bilinear form,
	\[s_h^{\rm sg}(\Pi_h^{\rm sg} u,\bm v_h)
	\le s_h^{\rm sg}(\Pi_h^{\rm sg} u,\Pi_h^{\rm sg} u)^{1/2}
	s_h^{\rm sg}(\bm v_h,\bm v_h)^{1/2}
	\le s_h^{\rm sg}(\Pi_h^{\rm sg} u,\Pi_h^{\rm sg} u)^{1/2}
	\norme{\bm v_h}.\]
    The jump term can be estimated by the trace inequality:
		\[\sum_{F\in \mathcal F^i_h}h^{-1}_F\|\jump{\PiCR u}\|^2_{0,F}=
		\sum_{F\in \mathcal F^i_h}h^{-1}_F\|\jump{\PiCR u-u}\|^2_{0,F}\lesssim h^2|u|^2_2.\]
	Using \(u=(u-u_0)+u_0\) and \eqref{eq:cr-commuting-general}
	\begin{align*}
		\sum_{F\in \mathcal F^i_h}h^{-1}_F\|\jump{\Pi_{\rm CR} u}\|^2_{0,F}
		= & \sum_{F\in\mathcal F^i_h}h^{-1}_F\|\jump{\PiCR u-u}\|^2_{0,F}\\
		\lesssim&
		\|u-u_0\|^2_1+\|(I-Q_0)\nabla u_0\|^2_0
		+\sum_{F\in \mathcal F^i_h}
		h^{-2}_F\|u_0-\PiCR u_0\|^2_{0,\omega_F}\\
        \lesssim & (\iota+h^2)\norm{f}^2_0.
	\end{align*}
	and 
	\begin{align*}
		\iota^2\sum_{T\in\mathcal T_h}h^{-2}_T\|\varepsilon_h(\PiCR u)-Q_0\PiCR \varepsilon (u)\|^2_{0,T}
		= & \iota^2\sum_{T\in\mathcal T_h}h^{-2}_T\|Q_0(\varepsilon (u)-\PiCR \varepsilon (u))\|^2_{0,T}\\
		\lesssim & \iota^2\sum_{T\in\mathcal T_h}h^{-2}_T\|(I-\PiCR) \varepsilon (u)\|^2_{0,T}\\
		\lesssim &\min\{\iota\norm{f}^2_0,h^{2s}|u|^2_{2+s}\}.
	\end{align*}
	Then, we have \eqref{eq:stabilization-error}.
    \end{proof}
	
	Finally, we consider estimating the consistency term. Similar to \eqref{eq:H1-consistency-estimate}, for any \(v_h\in \mathrm{CR}_{h0}\otimes \mathbb R^d\), \(\Pi_{\rm A} v_h\in A_{h0}\otimes \mathbb R^d\subset H^2_0(\Omega;\mathbb R^d)\), 
	\begin{equation}
		\label{eq:H1-consistency-estimate-v}
		\left(\sum_{T\in\mathcal T_h}h^{-2}_T\|v_h-\Pi_{\rm A}v_h\|_{0,T}^2\right)^{1/2}+\|\nabla_h(v_h-\Pi_{\rm A}v_h)\|_0\lesssim \|\nabla_h v_h\|_0.
	\end{equation}
	
	\begin{lemma}
		\label{lem:consistence-error-small-iota}
		Let \(u\in H^2_0(\Omega;\mathbb R^d)\) and \(u_0\in H^2(\Omega;\mathbb R^d)\cap H^1_0(\Omega;\mathbb R^d)\) be the solutions of \eqref{eq:sge-weak} and \eqref{eq:sge-LE}, respectively. Then
		\begin{equation}
			\label{eq:consistence-error-small-iota}
			\sup_{0\ne \bm v_h\in V_{h0}^{\rm sg}}
			\frac{|a_h^{\rm sg}(\Pi_h^{\rm sg}u,\bm v_h)-(f,v_h)|}{\norme{\bm v_h}}
			\lesssim \iota|\sigma(u)|_1
			+\|\sigma(u-u_0)\|_0
			+h\|f\|_0.
		\end{equation}
	\end{lemma}
	
	\begin{proof}
		By \eqref{eq:cr-commuting-sym}, \eqref{eq:sge-LE} and integration by parts, we obtain
		\begin{align*}
			a_h^{\rm sg}(\Pi_h^{\rm sg} u,\bm v_h)-(f,v_h)=\iota^2(\nabla\sigma(u),\nabla_h\psi_h)
			+(\sigma(u-u_0),\varepsilon_h(v_h))\\
			+(\sigma(u_0),\varepsilon_h(v_h-\Pi_{\rm A}v_h))-(f,v_h-\Pi_{\rm A}v_h).
		\end{align*}
		
		Inserting \(Q_0\sigma(u_0)\) gives
		\begin{equation}
			\label{eq:variational-error}
			(\sigma(u_0),\varepsilon_h(v_h-\Pi_{\rm A}v_h))=((I-Q_0)\sigma(u_0),\varepsilon_h(v_h-\Pi_{\rm A}v_h))
			+(Q_0\sigma(u_0),\varepsilon_h(v_h-\Pi_{\rm A}v_h)),
		\end{equation}
		the trace inequality, \eqref{eq:H1-consistency-estimate-v} and \eqref{eq:bih-fractional-means} imply
		\begin{align}
			\label{eq:jump-estimate}
			(Q_0\sigma(u_0),\varepsilon_h(v_h-\Pi_{\rm A}v_h))
			=&\sum_{F\in\mathcal F^i_h}(\jump{Q_0\sigma(u_0)}n_F,\{v_h-\Pi_{\rm A}v_h\})_F\notag\\
			\lesssim&\left(\sum_{F\in \mathcal F^i_h}h_F\|\jump{Q_0\sigma(u_0)}\|^2_{0,F}\right)^{1/2}
			\left(\sum_{F\in \mathcal F^i_h}h^{-1}_{F}
			\|\{v_h-\Pi_{\rm A}v_h\}\|^2_{0,F}\right)^{1/2}\notag\\
			\lesssim & h|\sigma(u_0)|_1\|\nabla_h v_h\|_0,
		\end{align}
		and 
		\[((I-Q_0)\sigma(u_0),\varepsilon_h(v_h-\Pi_{\rm A}v_h))
		\lesssim h|\sigma(u_0)|_1 \|\nabla_hv_h\|_0.\]
		
		Using \eqref{eq:sge-discrete-korn},
		\begin{align*}
			(\nabla\sigma(u),\nabla_h\psi_h)\lesssim &\iota^{-1}|\sigma(u)|_1\norme{\bm v_h},\\
			(\sigma(u-u_0),\varepsilon_h(v_h))\lesssim&\|\sigma(u-u_0)\|_0\norme{\bm v_h}\\
			(\sigma(u_0),\varepsilon_h(v_h-\Pi_{\rm A}v_h))\lesssim &
			h\norm{f}_0\norme{\bm v_h}.
		\end{align*}
		By \eqref{eq:H1-consistency-estimate-v} and \eqref{eq:sge-discrete-korn},
		\[(f,v_h-\Pi_{\rm A}v_h)\lesssim h \|f\|_0\norme{\bm v_h}.\]
	\end{proof}

	For large $\iota$, the consistency estimate under $u\in H^2(\Omega;\mathbb R^d)$ requires a componentwise Scott--Zhang-type quasi-interpolant $\Pi_{\rm SZ}$ from $L^2(\Omega;\mathbb S\otimes\mathbb R^d)$ onto the continuous piecewise linear space; cf.~\cite{scott1990finite}. We also need the following lemma.
	
	\begin{lemma}
		\label{lem:sge-H2-companion}
		For every
		\(v_h\in\mathrm{CR}_{h,0}\otimes\mathbb R^d\), 
		\begin{equation}
			\label{eq:sge-H2-companion}
			\sum_{j=0}^1
			\sum_{T\in \mathcal T_h}h^{2(j-1)}_T
			\|\nabla^j_h(v_h-\Pi_{\rm A}v_h)\|^2_{0,T}
			+\sum_{T\in \mathcal T_h}
			h_T^2\|\nabla^2\Pi_{\rm A} v_h\|_{0,T}^2
			\lesssim
			\|\varepsilon_h(v_h)\|_0^2
			+\sum_{F\in\mathcal F^i_h} h^{-1}_F\|\jump{v_h}\|^2_{0,F}.
		\end{equation}
	\end{lemma}
	
	\begin{proof}
		Since \(v_h|_T\in\mathbb P_1(T)\), the inverse inequality yields
		\[
		h_T^2\|\nabla^2\Pi_{\rm A} v_h\|_{0,T}^2
		=h_T^2\|\nabla^2(\Pi_{\rm A} v_h-v_h)\|_{0,T}^2
		\lesssim\|\nabla(\Pi_{\rm A} v_h-v_h)\|_{0,T}^2.
		\]
		Combining \eqref{eq:H1-consistency-estimate-v} with Lemma~\ref{lem:sge-discrete-korn}, we complete the proof.
	\end{proof}
	
	\begin{lemma}
		\label{lem:consistency-error-low-R}
		Let \(u\in H^2_0(\Omega;\mathbb R^d)\) be the solution of \eqref{eq:sge-weak}, with \(f\in L^2(\Omega;\mathbb R^d)\). Then 
		\begin{align}
			\sup_{0\ne \bm v_h\in V_{h0}^{\rm sg}}
			\frac{|a_h^{\rm sg}(\Pi_h^{\rm sg}u,\bm v_h)-(f, v_h)|}{\norme{\bm v_h}}
			\lesssim 
			\iota\|(I-\Pi_{\rm SZ})\nabla\sigma(u)\|_0
			+\iota^2\left(\sum_{T\in \mathcal T_h}h^{-2}_T\|(I-\Pi_{\rm SZ})\nabla\sigma(u)\|^2_{0,T}\right)^{1/2}\notag\\
			+\iota h|\Pi_{\rm SZ}\nabla\sigma(u)|_1
			+\iota^2\left(\sum_{F\in \mathcal F^i_h}h_F\|\jump{\operatorname{div}\Pi_{\rm SZ}\nabla\sigma(u)}\|^2_{0,F}\right)^{1/2}
			+h(|\sigma(u)|_1+\|f\|_0)
		\end{align}
	\end{lemma}
	
	\begin{proof}
		By \eqref{eq:sge-weak}, \eqref{eq:cr-commuting-sym} and integration by parts, we obtain
		\begin{align*}
			a_h^{\rm sg}(\Pi_h^{\rm sg} u,\bm v_h)-(f,v_h)
			=&\iota^2((I-\Pi_{\rm SZ})\nabla\sigma(u),\nabla_h(\psi_h-\varepsilon(\Pi_{\rm A} v_h)))\\
			&-\iota^2(\operatorname{div}\Pi_{\rm SZ}\nabla\sigma(u),\psi_h-\varepsilon(\Pi_{\rm A} v_h))
			+\iota^2\sum_{F\in \mathcal F_h}(\Pi_{\rm SZ}\nabla\sigma(u) n_F,\jump{\psi_h})_F\\
			&+(\sigma(u),\varepsilon_h(v_h-\Pi_{\rm A}v_h))-(f,v_h-\Pi_{\rm A}v_h).
		\end{align*}
		
		Using Lemma~\ref{lem:sge-H2-companion}, we have 
		\begin{align*}
			\iota^2((I-\Pi_{\rm SZ})\nabla\sigma(u),\nabla_h(\psi_h-\varepsilon(\Pi_{\rm A} v_h)))
			\lesssim \iota\|(I-\Pi_{\rm SZ})\nabla\sigma(u)\|_0\norme{\bm v_h}\\
			+\iota^2\left(\sum_{T\in \mathcal T_h}h^{-2}_T\|(I-\Pi_{\rm SZ})\nabla\sigma(u)\|^2_{0,T}\right)^{1/2}
			\norme{\bm v_h}.
		\end{align*}
		Inserting \(\varepsilon_h(v_h)\) into the second term
		\begin{align*}
			\iota^2(\operatorname{div}\Pi_{\rm SZ}\nabla\sigma(u),\psi_h-\varepsilon(\Pi_{\rm A} v_h))
			=&\iota^2(\operatorname{div}\Pi_{\rm SZ}\nabla\sigma(u),\psi_h-\varepsilon_h( v_h))\\
			&+\iota^2(\operatorname{div}\Pi_{\rm SZ}\nabla\sigma(u),\varepsilon_h( v_h)-\varepsilon(\Pi_{\rm A} v_h)),
		\end{align*}
		where \(\operatorname{div}\Pi_{\rm SZ}\nabla\sigma(u)\) is piecewise‑constant,
		\begin{align*}
			\iota^2(\operatorname{div}\Pi_{\rm SZ}\nabla\sigma(u),\psi_h-\varepsilon_h(v_h))=&\iota^2(\operatorname{div}\Pi_{\rm SZ}\nabla\sigma(u),Q_0\psi_h-\varepsilon_h(v_h))\\
			\lesssim &\iota h|\Pi_{\rm SZ}\nabla\sigma(u)|_1\norme{\bm v_h},
		\end{align*}
		with Lemma~\ref{lem:sge-discrete-korn}, we have 
		\begin{align*}
			\iota^2(\operatorname{div}\Pi_{\rm SZ}\nabla\sigma(u),\varepsilon_h( v_h)-\varepsilon(\Pi_{\rm A} v_h))
			=&\iota^2\sum_{F\in\mathcal F^i_h}(\jump{\operatorname{div}\Pi_{\rm SZ}\nabla\sigma(u)} n_F,\{v_h-\Pi_{\rm A} v_h\})_F\\
			\lesssim & \iota^2\left(\sum_{F\in \mathcal F^i_h}h_F\|\jump{\operatorname{div}\Pi_{\rm SZ}\nabla\sigma(u)}\|^2_{0,F}\right)^{1/2}\norme{\bm v_h}.
		\end{align*}
		
		Using the trace inequality, 
		\begin{align*}
			\iota^2\sum_{F\in \mathcal F_h}(\Pi_{\rm SZ}\nabla\sigma(u) n_F,\jump{\psi_h})_F
			=&\iota^2\sum_{F\in \mathcal F_h}((I-Q_{\omega_F})\Pi_{\rm SZ}\nabla\sigma(u) n_F,\jump{\psi_h})_F\\
			\lesssim &\iota^2 \sum_{F\in \mathcal F_h} h_F|\Pi_{\rm SZ}\nabla\sigma(u)|_{1,\omega_F}\|\nabla_h\psi_h\|_{0,\omega_F}\\
			\lesssim &\iota h |\Pi_{\rm SZ}\nabla\sigma(u)|_1\norme{\bm v_h}, 
		\end{align*}
		where $Q_{\omega_F}$ denotes the constant $L^2$-projection onto the patch $\omega_F$.
		Similar to the proof of \eqref{eq:variational-error} and \eqref{eq:sge-H2-companion}, 
		\[(\sigma(u),\varepsilon_h(v_h-\Pi_{\rm A}v_h))\lesssim  h|\sigma(u)|_1\norme{\bm v_h}.\]
		Finally, using \eqref{eq:sge-H2-companion}
		\[(f,v_h-\Pi_{\rm A}v_h)\lesssim h\|f\|_0\norme{\bm v_h}.\]
	\end{proof}
	
	\begin{remark}
		Assume \(u\in H^4(\Omega;\mathbb R^d)\). Under Assumptions~\ref{ass:le-regularity} and~\ref{ass:sge-regularity}, Lemmas~\ref{lem:consistence-error-small-iota} and~\ref{lem:consistency-error-low-R} yield
        \begin{align*}
            \sup_{0\ne \bm v_h\in V_{h0}^{\rm sg}}
            \frac{|a_h^{\rm sg}(\Pi_h^{\rm sg}u,\bm v_h)-(f,v_h)|}
            {\norme{\bm v_h}}\lesssim
            \min \{
            \iota^2 h|\sigma(u)|_3
			+\iota h|\sigma(u)|_2
            +h|\sigma(u)|_1+h\|f\|_0,\\
            (\iota^{1/2}+h)\norm{f}_0\}.
        \end{align*}
	\end{remark}

    \begin{theorem}
        Let \(u\in H^2_0(\Omega;\mathbb R^d)\cap H^4(\Omega;\mathbb R^d)\) and \(\bm u_h\in V^{\rm sg}_{h0}\) be the solutions of \eqref{eq:sge-weak} and \eqref{eq:sge-discrete}, respectively. Under Assumption~\ref{ass:le-regularity} and Assumption~\ref{ass:sge-regularity}, we have 
        \begin{align}
            |(u,\varepsilon(u))-\bm u_h|_{\rm E}
            \lesssim \min \{
            \iota^2 h|\sigma(u)|_3
			+\iota h|\sigma(u)|_2
            +h|\sigma(u)|_1+h\|f\|_0,\notag\\
            (\iota^{1/2}+h)\norm{f}_0\}.
        \end{align}
        The hidden constant is independent of \(h\), \(\iota\) and \(\lambda\).
    \end{theorem}

	\section{Numerical experiments}
	\label{sec:numerical-experiments}
	We first examine the convergence of \eqref{eq:sge-discrete} and its robustness
	with respect to the size parameter \(\iota\) and the Lam\'e constant \(\lambda\).
	Let \(\Omega=(0,1)^2\), set \(\mu=1\), and use uniformly refined triangular
	meshes. In the first two tests, the prescribed displacement is divergence-free;
	hence the right-hand side \(f\), computed from \eqref{eq:sge-model}, is
	independent of \(\lambda\). We report the relative energy error
	\[
	\operatorname{Err}_{\rm E}
	:=
	\frac{|(u,\varepsilon(u))-\bm u_h|_{\rm E}}
	{|(u,\varepsilon(u))|_{\rm E}}.
	\]
	
	The first test uses the following smooth solution without boundary layers:
	\[
	u(x,y)
	=\operatorname{curl}(\sin^3(\pi x)\sin^3(\pi y))
	=
	\begin{pmatrix}
		3\pi \sin^3(\pi x)\sin^2(\pi y)\cos(\pi y)\\[1mm]
		-3\pi \sin^3(\pi y)\sin^2(\pi x)\cos(\pi x)
	\end{pmatrix},
	\]
	for which \(\operatorname{div}u=0\). Tables~\ref{tab:divfree-smooth-lam1}
	and~\ref{tab:divfree-smooth-lam1e6} list \(\operatorname{Err}_{\rm E}\) for
	\(\lambda=1\) and \(10^6\), respectively.
	
	\begin{table}[htbp]
		\centering
		\caption{Relative energy errors \(\operatorname{Err}_{\rm E}\) for the smooth solution with \(\lambda=1\).}
		\label{tab:divfree-smooth-lam1}
		\begin{tabular}{c|cccccc}
			\hline
			$\iota\backslash h$ & $1/2$ & $1/4$ & $1/8$ & $1/16$ & $1/32$ & $1/64$ \\
			\hline
			$1$        & $3.216{\rm e}{+0}$ & $1.253{\rm e}{+0}$ & $6.585{\rm e}{-1}$ & $4.921{\rm e}{-1}$ & $3.355{\rm e}{-1}$ & $1.682{\rm e}{-1}$ \\
			rate       &                    & $1.36$              & $0.93$              & $0.42$              & $0.55$              & $1.00$              \\
			$10^{-2}$  & $8.501{\rm e}{-1}$ & $6.340{\rm e}{-1}$ & $3.254{\rm e}{-1}$ & $1.607{\rm e}{-1}$ & $7.897{\rm e}{-2}$ & $3.861{\rm e}{-2}$ \\
			rate       &                    & $0.42$              & $0.96$              & $1.02$              & $1.03$              & $1.03$              \\
			$10^{-4}$  & $8.450{\rm e}{-1}$ & $6.284{\rm e}{-1}$ & $3.195{\rm e}{-1}$ & $1.578{\rm e}{-1}$ & $7.858{\rm e}{-2}$ & $3.924{\rm e}{-2}$ \\
			rate       &                    & $0.43$              & $0.98$              & $1.02$              & $1.01$              & $1.00$              \\
			$10^{-6}$  & $8.450{\rm e}{-1}$ & $6.284{\rm e}{-1}$ & $3.195{\rm e}{-1}$ & $1.578{\rm e}{-1}$ & $7.858{\rm e}{-2}$ & $3.924{\rm e}{-2}$ \\
			rate       &                    & $0.43$              & $0.98$              & $1.02$              & $1.01$              & $1.00$              \\
			$10^{-8}$  & $8.450{\rm e}{-1}$ & $6.284{\rm e}{-1}$ & $3.195{\rm e}{-1}$ & $1.578{\rm e}{-1}$ & $7.858{\rm e}{-2}$ & $3.924{\rm e}{-2}$ \\
			rate       &                    & $0.43$              & $0.98$              & $1.02$              & $1.01$              & $1.00$              \\
			\hline
		\end{tabular}
	\end{table}
	
	\begin{table}[htbp]
		\centering
		\caption{Relative energy errors \(\operatorname{Err}_{\rm E}\) for the smooth solution with \(\lambda=10^6\).}
		\label{tab:divfree-smooth-lam1e6}
		\begin{tabular}{c|cccccc}
			\hline
			$\iota\backslash h$ & $1/2$ & $1/4$ & $1/8$ & $1/16$ & $1/32$ & $1/64$ \\
			\hline
			$1$        & $2.868{\rm e}{+0}$ & $1.112{\rm e}{+0}$ & $6.451{\rm e}{-1}$ & $4.896{\rm e}{-1}$ & $3.356{\rm e}{-1}$ & $1.717{\rm e}{-1}$ \\
			rate       &                    & $1.37$              & $0.79$              & $0.40$              & $0.54$              & $0.97$              \\
			$10^{-2}$  & $8.538{\rm e}{-1}$ & $6.279{\rm e}{-1}$ & $3.208{\rm e}{-1}$ & $1.580{\rm e}{-1}$ & $7.747{\rm e}{-2}$ & $3.795{\rm e}{-2}$ \\
			rate       &                    & $0.44$              & $0.97$              & $1.02$              & $1.03$              & $1.03$              \\
			$10^{-4}$  & $8.488{\rm e}{-1}$ & $6.221{\rm e}{-1}$ & $3.148{\rm e}{-1}$ & $1.550{\rm e}{-1}$ & $7.702{\rm e}{-2}$ & $3.843{\rm e}{-2}$ \\
			rate       &                    & $0.45$              & $0.98$              & $1.02$              & $1.01$              & $1.00$              \\
			$10^{-6}$  & $8.488{\rm e}{-1}$ & $6.221{\rm e}{-1}$ & $3.148{\rm e}{-1}$ & $1.550{\rm e}{-1}$ & $7.702{\rm e}{-2}$ & $3.843{\rm e}{-2}$ \\
			rate       &                    & $0.45$              & $0.98$              & $1.02$              & $1.01$              & $1.00$              \\
			$10^{-8}$  & $8.488{\rm e}{-1}$ & $6.221{\rm e}{-1}$ & $3.148{\rm e}{-1}$ & $1.550{\rm e}{-1}$ & $7.702{\rm e}{-2}$ & $3.843{\rm e}{-2}$ \\
			rate       &                    & $0.45$              & $0.98$              & $1.02$              & $1.01$              & $1.00$              \\
			\hline
		\end{tabular}
	\end{table}
	
	On the refined meshes, the errors converge at approximately first order for
	all tested values of \(\iota\). They do not deteriorate as \(\iota\) decreases
	and are nearly identical for \(\iota\le10^{-4}\); changing \(\lambda\) from
	\(1\) to \(10^6\) also has negligible effect. This supports the expected
	parameter-robust convergence and locking-free behavior for this
	divergence-free test.
	
	The second test uses the following \(\iota\)-dependent solution with boundary
	layers:
	\[
	u(x,y)
	=\operatorname{curl}(g_1(x)^2g_1(y)^2)
	=
	\begin{pmatrix}
		2g_1(x)^2 g_1(y)g_1'(y)\\[1mm]
		-2g_1(x)g_1'(x)g_1(y)^2
	\end{pmatrix},
	\]
	where
	\[
	g_1(x)
	=
	e^{\sin(\pi x)}-1
	-
	\pi\iota
	\frac{
		1+e^{-1/\iota}-e^{-x/\iota}-e^{-(1-x)/\iota}
	}{
		1-e^{-1/\iota}
	}.
	\]
	This displacement satisfies \(\operatorname{div}u=0\) and the clamped boundary
	conditions, while the exponential terms in \(g_1\) generate boundary layers near
	\(\partial\Omega\) as \(\iota\to0\). Tables~\ref{tab:divfree-layer-lam1}
	and~\ref{tab:divfree-layer-lam1e6} report \(\operatorname{Err}_{\rm E}\). For \(\lambda=1\), Figure~\ref{fig:boundary-layers}
	shows the six independent components of \(\nabla\varepsilon(u)\) for
	\(\iota=10^{-2}\). 
	
	\begin{figure}[htbp]
		\centering
		\includegraphics[width=0.75\textwidth]{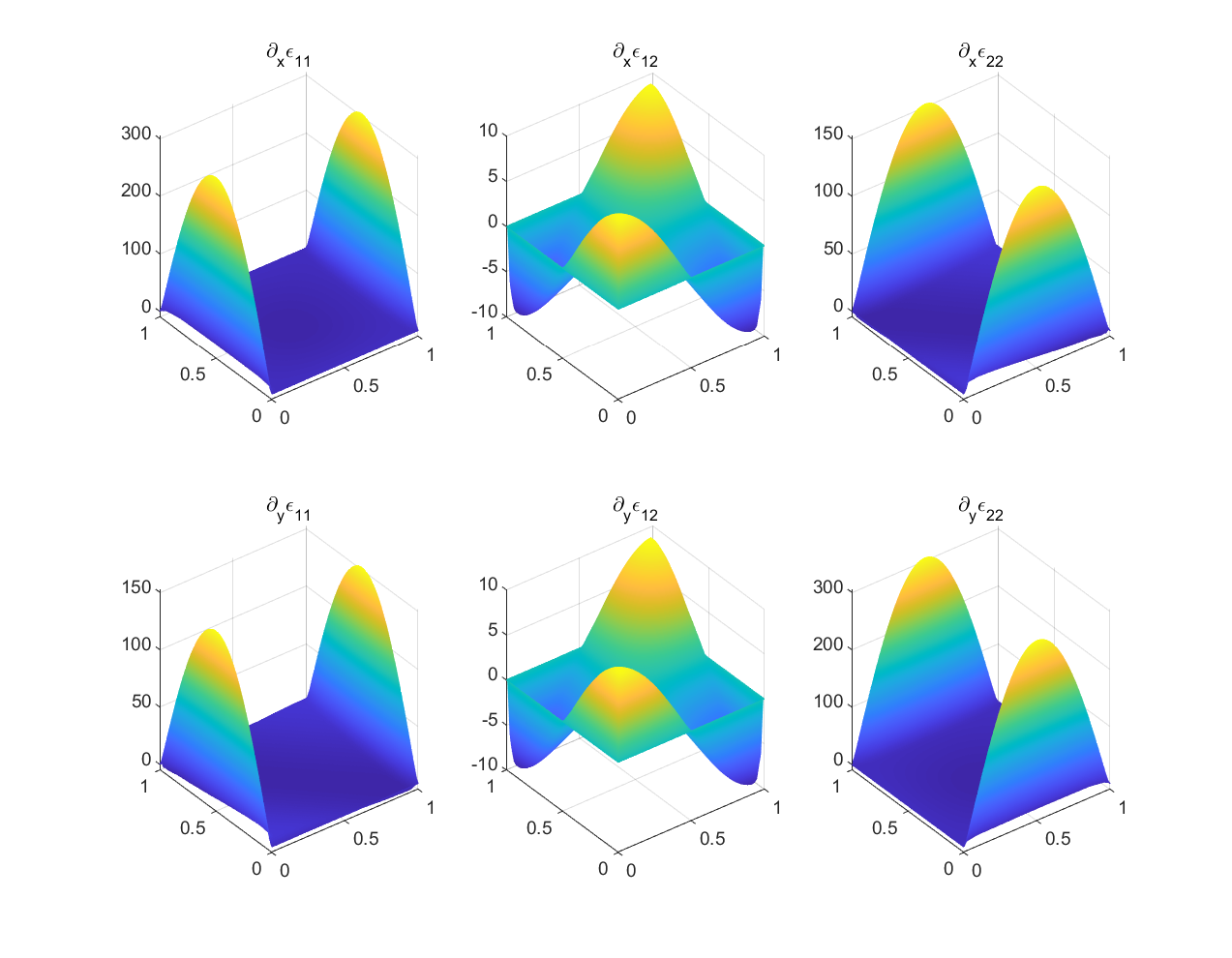}
		\caption{
			Boundary-layer profiles of the six independent components of
			\(\nabla\varepsilon(u)\) for the divergence-free
			boundary-layer solution with $\iota=10^{-2}$.
		}
		\label{fig:boundary-layers}
	\end{figure}
	
	\begin{table}[H]
		\centering
		\caption{Relative energy errors \(\operatorname{Err}_{\rm E}\) for the boundary-layer solution with \(\lambda=1\).}
		\label{tab:divfree-layer-lam1}
		\begin{tabular}{c|cccccc}
			\hline
			$\iota\backslash h$ & $1/2$ & $1/4$ & $1/8$ & $1/16$ & $1/32$ & $1/64$ \\
			\hline
			$1$        & $3.430{\rm e}{+0}$ & $1.809{\rm e}{+0}$ & $8.989{\rm e}{-1}$ & $6.999{\rm e}{-1}$ & $5.002{\rm e}{-1}$ & $2.620{\rm e}{-1}$ \\
			rate       &                    & $0.92$              & $1.01$              & $0.36$              & $0.48$              & $0.93$              \\
			$10^{-2}$  & $8.123{\rm e}{-1}$ & $6.267{\rm e}{-1}$ & $3.216{\rm e}{-1}$ & $1.639{\rm e}{-1}$ & $8.802{\rm e}{-2}$ & $4.947{\rm e}{-2}$ \\
			rate       &                    & $0.37$              & $0.96$              & $0.97$              & $0.90$              & $0.83$              \\
			$10^{-4}$  & $7.990{\rm e}{-1}$ & $6.115{\rm e}{-1}$ & $3.036{\rm e}{-1}$ & $1.501{\rm e}{-1}$ & $7.479{\rm e}{-2}$ & $3.736{\rm e}{-2}$ \\
			rate       &                    & $0.39$              & $1.01$              & $1.02$              & $1.01$              & $1.00$              \\
			$10^{-6}$  & $7.989{\rm e}{-1}$ & $6.115{\rm e}{-1}$ & $3.036{\rm e}{-1}$ & $1.501{\rm e}{-1}$ & $7.478{\rm e}{-2}$ & $3.735{\rm e}{-2}$ \\
			rate       &                    & $0.39$              & $1.01$              & $1.02$              & $1.00$              & $1.00$              \\
			$10^{-8}$  & $7.989{\rm e}{-1}$ & $6.115{\rm e}{-1}$ & $3.036{\rm e}{-1}$ & $1.501{\rm e}{-1}$ & $7.478{\rm e}{-2}$ & $3.735{\rm e}{-2}$ \\
			rate       &                    & $0.39$              & $1.01$              & $1.02$              & $1.00$              & $1.00$              \\
			\hline
		\end{tabular}
	\end{table}
	
	\begin{table}[H]
		\centering
		\caption{Relative energy errors \(\operatorname{Err}_{\rm E}\) for the boundary-layer solution with \(\lambda=10^6\).}
		\label{tab:divfree-layer-lam1e6}
		\begin{tabular}{c|cccccc}
			\hline
			$\iota\backslash h$ & $1/2$ & $1/4$ & $1/8$ & $1/16$ & $1/32$ & $1/64$ \\
			\hline
			$1$        & $3.422{\rm e}{+0}$ & $1.571{\rm e}{+0}$ & $8.541{\rm e}{-1}$ & $6.852{\rm e}{-1}$ & $4.985{\rm e}{-1}$ & $2.686{\rm e}{-1}$ \\
			rate       &                    & $1.12$              & $0.88$              & $0.32$              & $0.46$              & $0.89$              \\
			$10^{-2}$  & $8.152{\rm e}{-1}$ & $6.205{\rm e}{-1}$ & $3.166{\rm e}{-1}$ & $1.612{\rm e}{-1}$ & $8.640{\rm e}{-2}$ & $4.837{\rm e}{-2}$ \\
			rate       &                    & $0.39$              & $0.97$              & $0.97$              & $0.90$              & $0.84$              \\
			$10^{-4}$  & $8.021{\rm e}{-1}$ & $6.052{\rm e}{-1}$ & $2.987{\rm e}{-1}$ & $1.473{\rm e}{-1}$ & $7.325{\rm e}{-2}$ & $3.657{\rm e}{-2}$ \\
			rate       &                    & $0.41$              & $1.02$              & $1.02$              & $1.01$              & $1.00$              \\
			$10^{-6}$  & $8.021{\rm e}{-1}$ & $6.052{\rm e}{-1}$ & $2.987{\rm e}{-1}$ & $1.473{\rm e}{-1}$ & $7.325{\rm e}{-2}$ & $3.657{\rm e}{-2}$ \\
			rate       &                    & $0.41$              & $1.02$              & $1.02$              & $1.01$              & $1.00$              \\
			$10^{-8}$  & $8.021{\rm e}{-1}$ & $6.052{\rm e}{-1}$ & $2.987{\rm e}{-1}$ & $1.473{\rm e}{-1}$ & $7.325{\rm e}{-2}$ & $3.657{\rm e}{-2}$ \\
			rate       &                    & $0.41$              & $1.02$              & $1.02$              & $1.01$              & $1.00$              \\
			\hline
		\end{tabular}
	\end{table}
	
	Despite the boundary layers, the errors approach first-order convergence on
	the refined meshes. In particular, the results for
	\(\iota=10^{-4},10^{-6},10^{-8}\) are nearly identical, and increasing
	\(\lambda\) from \(1\) to \(10^6\) has no visible adverse effect. The test
	therefore supports the robustness of \eqref{eq:sge-discrete} with respect to
	both parameters in the boundary-layer regime.
	
	We consider the strain-gradient Kirsch problem in
	\cite{khakalo2017gradient} to investigate the size-dependent stress
	concentration around a circular hole. The original problem is posed on an infinite plate containing a circular
	hole of radius \(a\). There is no body force, so \(f=0\). The plate is subjected to a uniform remote tensile stress \(\sigma_\infty\) in the vertical direction. We truncate the plate to the quarter-domain
	\[
	\Omega=(0,L)^2\setminus\overline{B_a},
	\qquad
	B_a:=\{x\in\mathbb R^2:|x|<a\},
	\qquad
	\frac{L}{a}=100;
	\]
	see Figure~\ref{fig:kirsch-mesh}. To resolve the large stress variations near \(\partial B_a\) without an
	excessive number of degrees of freedom, we use a graded triangular mesh that
	is fine along the circular boundary and becomes progressively coarser toward
	the outer boundaries. The computational domain is divided into 18473 elements, with the mesh size set to \(0.001a\) near the hole and increased away from the hole with a growth rate of \(1.1\).
	\begin{figure}[htbp]
		\centering
		\includegraphics[width=0.78\textwidth]{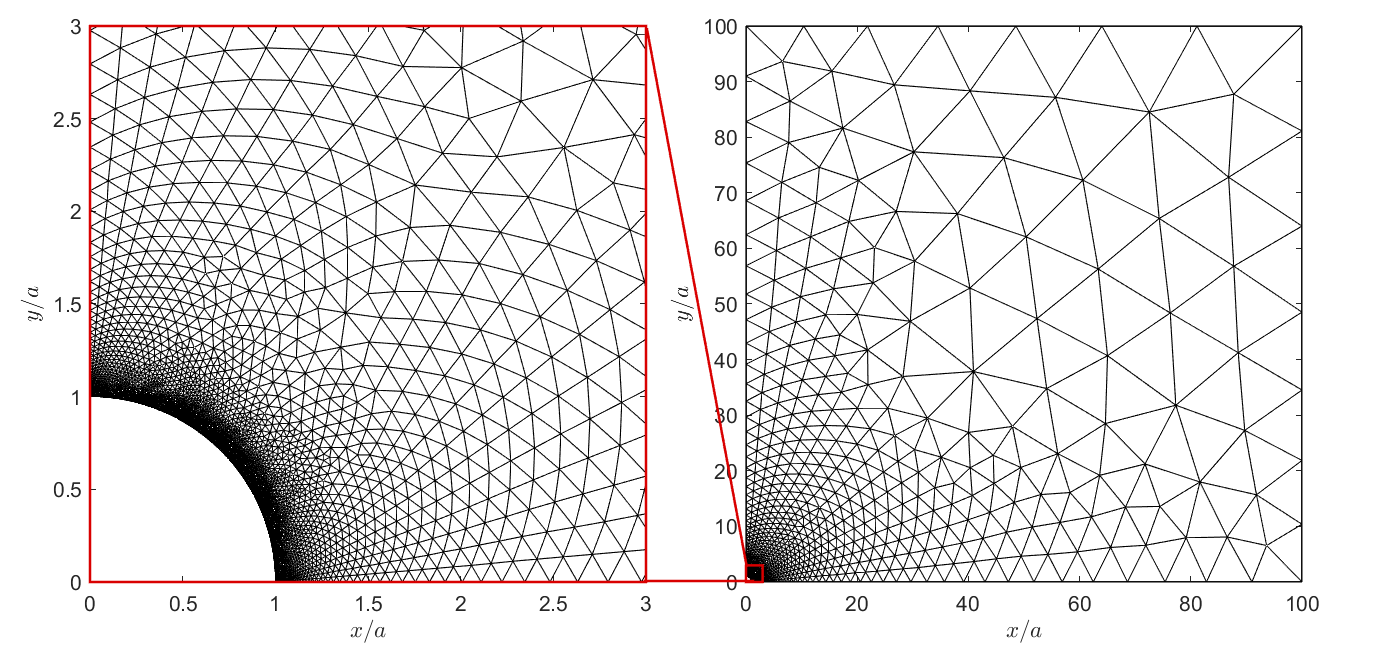}
		\caption{Graded mesh for the truncated quarter-domain: the near-hole
			region (left) and the full computational domain (right).}
		\label{fig:kirsch-mesh}
	\end{figure}
	
	
	On \((\{0\}\times(a,L))\cup((a,L)\times\{0\})\), we impose the symmetry conditions \(u\cdot n=0\) and \(\partial_n(u\cdot t)=0\), where \(t\) and \(n\) denote the unit tangent and unit outward normal, respectively. Since \(u\cdot n=0\) there, we have \(\partial_t(u\cdot n)=0\). Then \(2n^T\varepsilon(u)t=\partial_n(u\cdot t)+\partial_t(u\cdot n)=0.\) Natural boundary conditions are imposed on the remaining boundaries. To test the size effect
	, we introduce the stress concentration factor. Let \(\sigma:=\sigma(u)\). Then
	\[
	\begin{aligned}
		\sigma=&\sigma_{xx} (e_x\otimes e_x)+\sigma_{xy} (e_x\otimes e_y+e_y\otimes e_x)+\sigma_{yy} (e_y\otimes e_y)\\
	\end{aligned}\]
	where $e_x$ and $e_y$ denote the unit vectors along the $x$- and $y$-axes, respectively. We
	define
	\[
	K_t:=\frac{\sigma_{yy}(a,0)}{\sigma_{\infty}}.
	\]
	For the classical Kirsch problem, \(K_t=3\), independently of \(E\) and
	\(\nu\).
	
	Accordingly, we introduce the discrete space
	\[
	\begin{aligned}
		V_h^{\rm sg,K}
		:=\biggl\{
		(v_h,\psi_h)\in V_h^{\rm sg}
		:\int_F v_h\cdot n\mathrm{d}s=0,\quad\int_Fn^T\psi_ht\mathrm{d}s=0,\\
		\forall F\in \mathcal F_h,~F\subset \bigl(\{0\}\times(a,L)\cup (a,L)\times \{0\}\bigr).
		\biggr\}.
	\end{aligned}
	\]
	We then discretize the Kirsch problem by the CR--CR scheme: find \(\bm u_h\in V_h^{\rm sg,K}\) such that
	\[a_h^{\rm sg}(\bm u_h,\bm v_h)
	+
	s_h^{\rm sg}(\bm u_h,\bm v_h)
	=
	\int_{(0,L)\times\{L\}} (\sigma_\infty n)\cdot v_h\mathrm{d}s,
	\qquad
	\forall
	\bm v_h=(v_h,\psi_h)\in V_h^{\rm sg,K}.\]
	It is easy to verify that $\norme{\cdot}$ defines a norm on the space $V_h^{\rm sg,K}$, and hence the method is well-posed.
	
	Let Young's modulus \(E=210~\rm{GPa}\), Poisson's ratio \(\nu=0.3\) and remote tensile stress \(\sigma_\infty=0.1~\rm{GPa}\). The Lam\'e parameters appearing in the plane-stress setting are consequently
	\[
	\mu=\frac{E}{2(1+\nu)}
	\approx80.7692\ {\rm GPa},
	\qquad
	\lambda=\frac{E\nu}{1-\nu^2}
	\approx69.2308\ {\rm GPa}.
	\]
	
	Table~\ref{tab:kirsch-truncation} reports the computed \(K_t\) for \(\iota/a=10^{-6}\) with different truncation sizes
	\(L/a\). The purpose is to examine the influence of replacing the infinite
	plate by a finite computational domain. As shown in
	Table~\ref{tab:kirsch-truncation}, the value of \(K_t\) becomes essentially
	insensitive to the truncation size for \(L/a\ge100\). Therefore,
	\(L/a=100\) is adopted in the following computations.
	
	\begin{table}[htbp]
		\centering
		\caption{Effect of the truncation size on the stress concentration factor
			for \(\iota/a=10^{-6}\).}
		\label{tab:kirsch-truncation}
		\begin{tabular}{c|cccccccc}
			\hline
			\(L/a\) & 5 & 10 & 25 & 50 & 100 & 150 & 200 & 300 \\
			\hline
			\(K_t\) &
			3.350 &
			3.079 &
			3.006 &
			2.997 &
			2.994 &
			2.993 &
			2.993 &
			2.993 \\
			\hline
		\end{tabular}
	\end{table}

	Figure~\ref{fig:kirsch-Kt} compares the analytical and CR--CR values
	of \(K_t\) for different values of \(\iota/a\); let \(\sigma_h=2\mu\phi_h+\lambda\operatorname{tr}(\phi_h)I\), the discrete value at
	\((a,0)\) is obtained by local recovery from the neighboring CR elements. The stress concentration factor \(K_t\) decreases from its classical limit \(3\) as \(\iota/a\) increases. In particular, for \(\iota/a=10^{-6}\), the CR--CR solution gives \(K_t=2.99410520\) at \((a,0)\), compared with the classical value \(3\).
	
	\begin{figure}[htbp]
		\centering
		\includegraphics[width=0.62\textwidth]{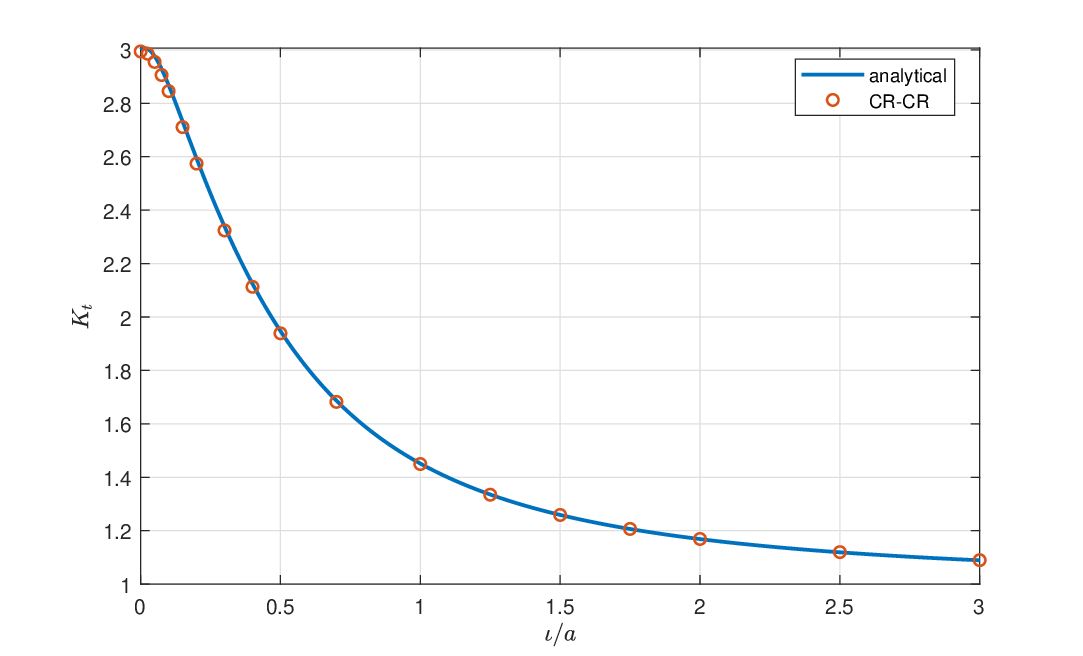}
		\caption{Stress concentration factor \(K_t\) versus \(\iota/a\). The solid
			line denotes the analytical solution and the markers denote the CR--CR
			approximation.}
		\label{fig:kirsch-Kt}
	\end{figure}
	
	We next compare the normalized stress components \(\sigma_{xx}/\sigma_{\infty}\) and
	\(\sigma_{yy}/\sigma_{\infty}\) along the positive
	\(x\)-axis: \(y=0\) and \(1\le x/a\le3\), for \(\iota/a=10^{-6},0.5,1\). In Figure~\ref{fig:kirsch-line-stress}, the analytical profiles are shown by continuous curves, while the CR--CR values are evaluated directly from the boundary-edge degrees of freedom.  The numerical results agree well with the analytical profiles and capture the pronounced stress variation near the hole.
	
	\begin{figure}[htbp]
		\centering
		\includegraphics[width=\textwidth]{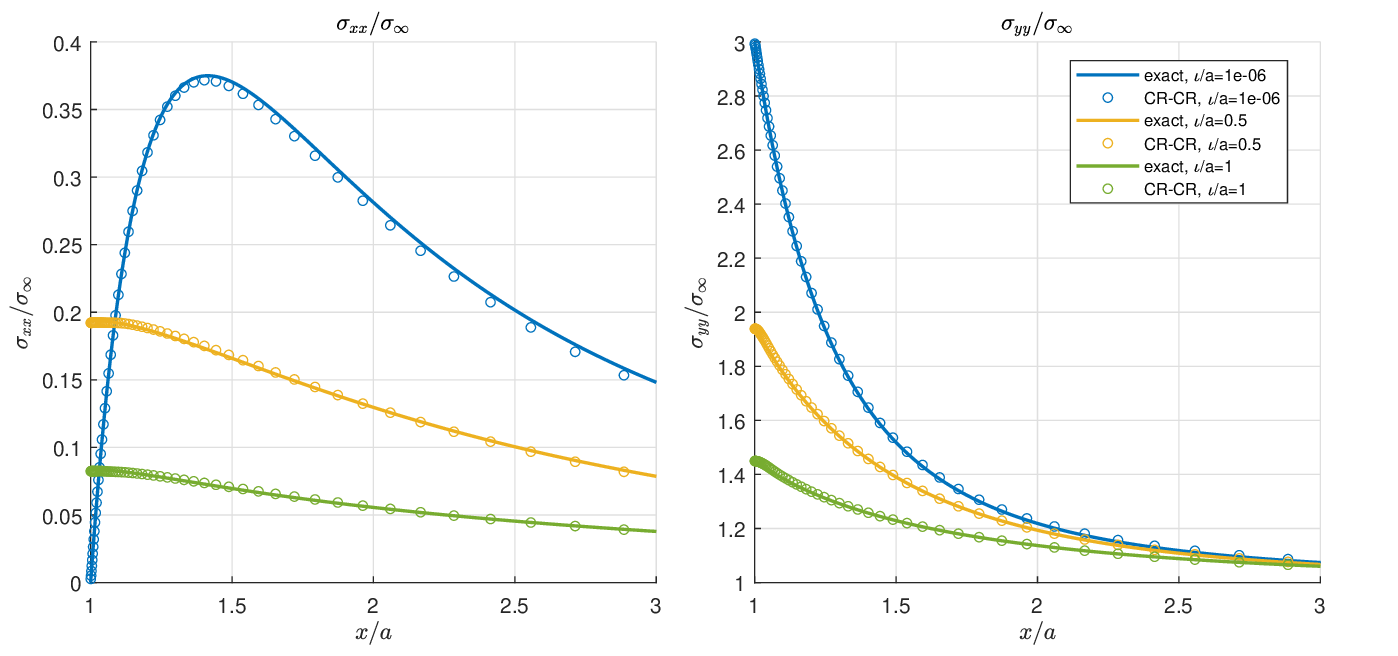}
		\caption{Normalized stress components along the positive \(x\)-axis for
			\(\iota/a=10^{-6},0.5,1\). From left to right:
			\(\sigma_{xx}/\sigma_{\infty}\) and
			\(\sigma_{yy}/\sigma_{\infty}\). Solid lines represent the analytical
			solution and markers represent the CR--CR approximation.}
		\label{fig:kirsch-line-stress}
	\end{figure}
	
	For each component, the analytical and numerical fields are shown for
	\(\iota/a=10^{-6},0.5,1\); see Figures~\ref{fig:kirsch-sxx}--\ref{fig:kirsch-syy}.
	Each column corresponds to one value of \(\iota/a\), with the analytical
	solution, the CR--CR solution, and their absolute error displayed from top
	to bottom.  The numerical results are in good agreement with the analytical solutions and clearly reproduce the size effect.
	
	\begin{figure}[htbp]
		\centering
		\includegraphics[width=\textwidth]{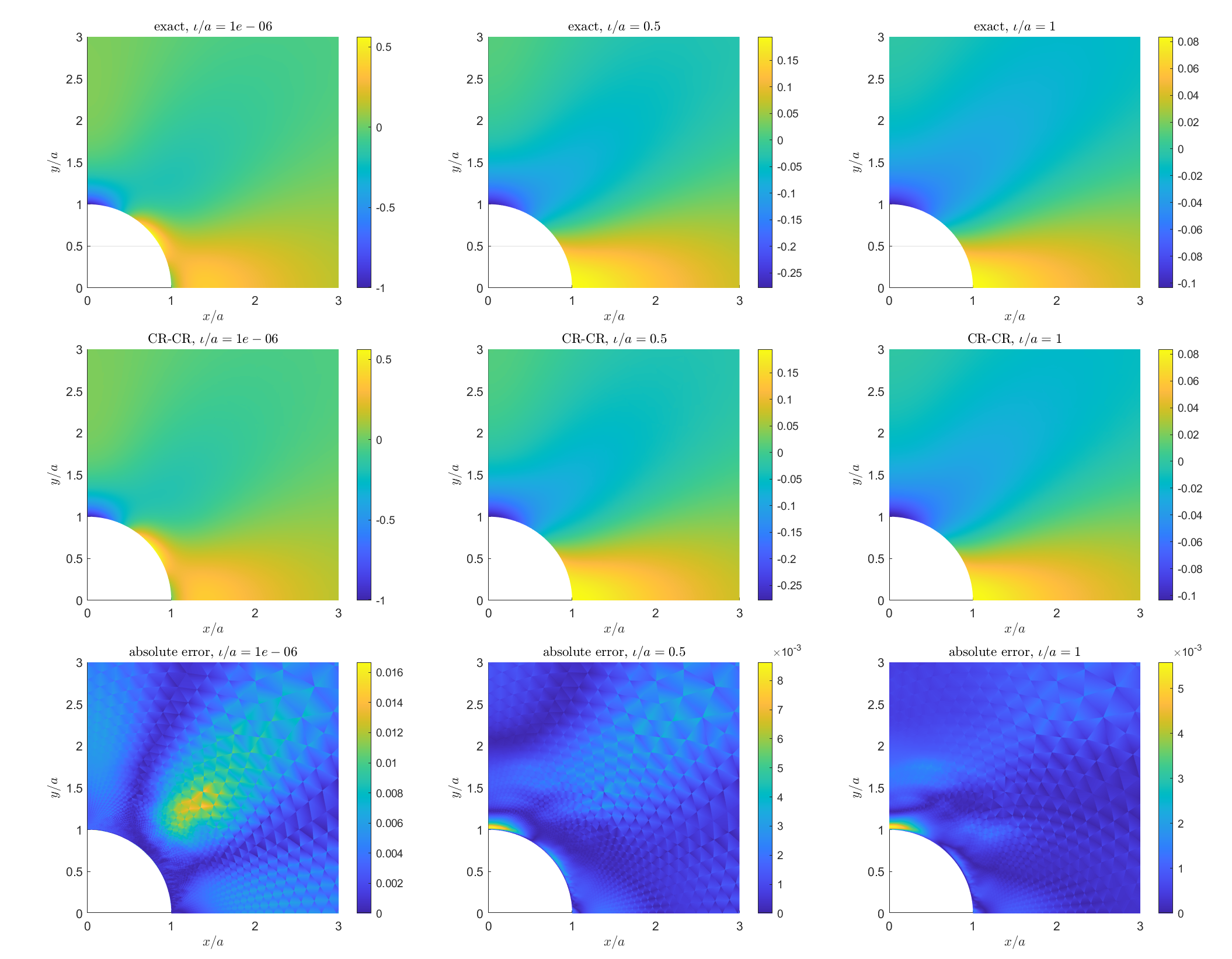}
		\caption{Normalized stress \(\sigma_{xx}/\sigma_{\infty}\) for
			\(\iota/a=10^{-6},0.5,1\): analytical and CR--CR solutions and absolute error.}
		\label{fig:kirsch-sxx}
	\end{figure}
	
	\begin{figure}[htbp]
		\centering
		\includegraphics[width=\textwidth]{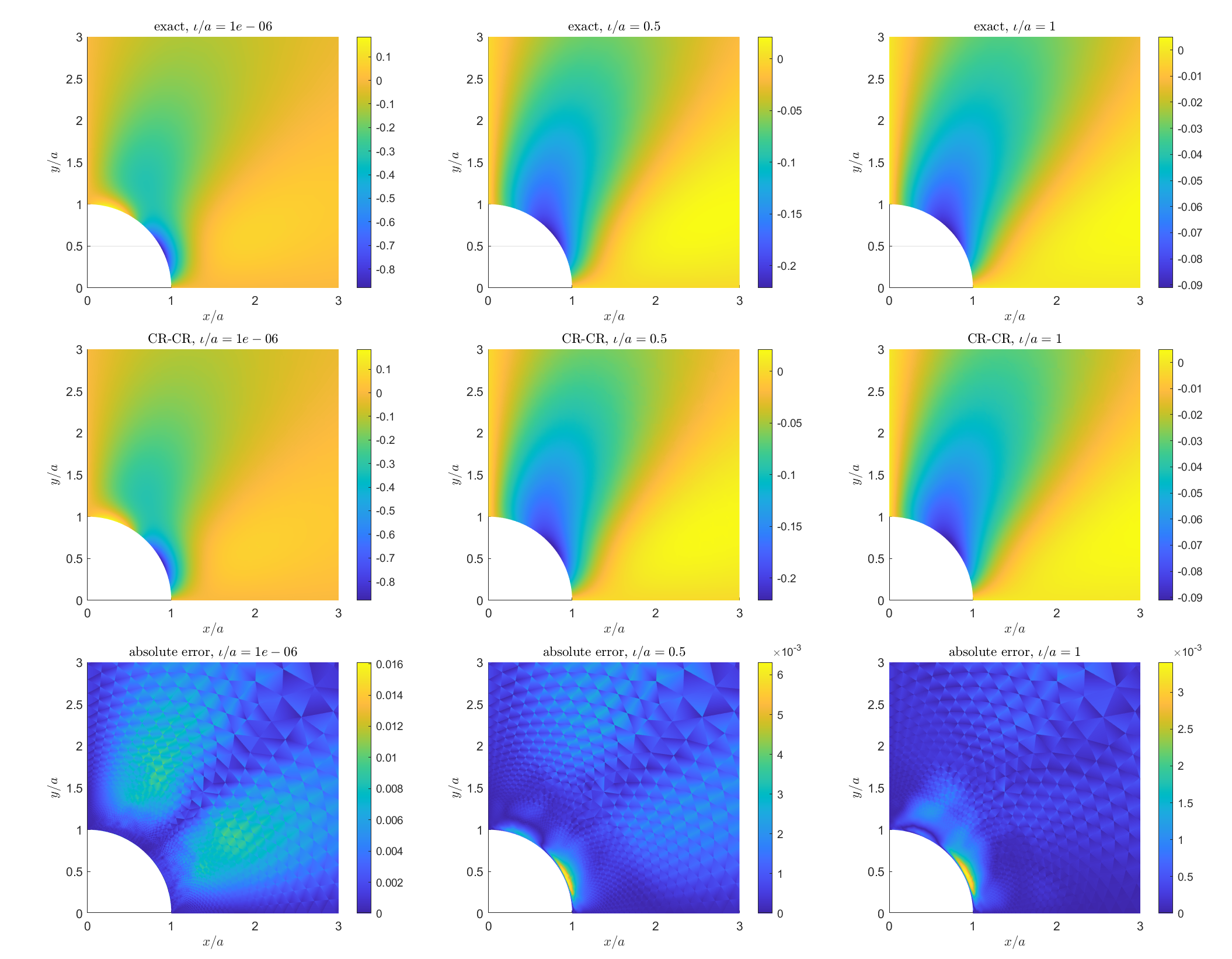}
		\caption{Normalized stress \(\sigma_{xy}/\sigma_{\infty}\) for
			\(\iota/a=10^{-6},0.5,1\): analytical and CR--CR solutions and absolute error.}
		\label{fig:kirsch-sxy}
	\end{figure}
	
	\begin{figure}[htbp]
		\centering
		\includegraphics[width=\textwidth]{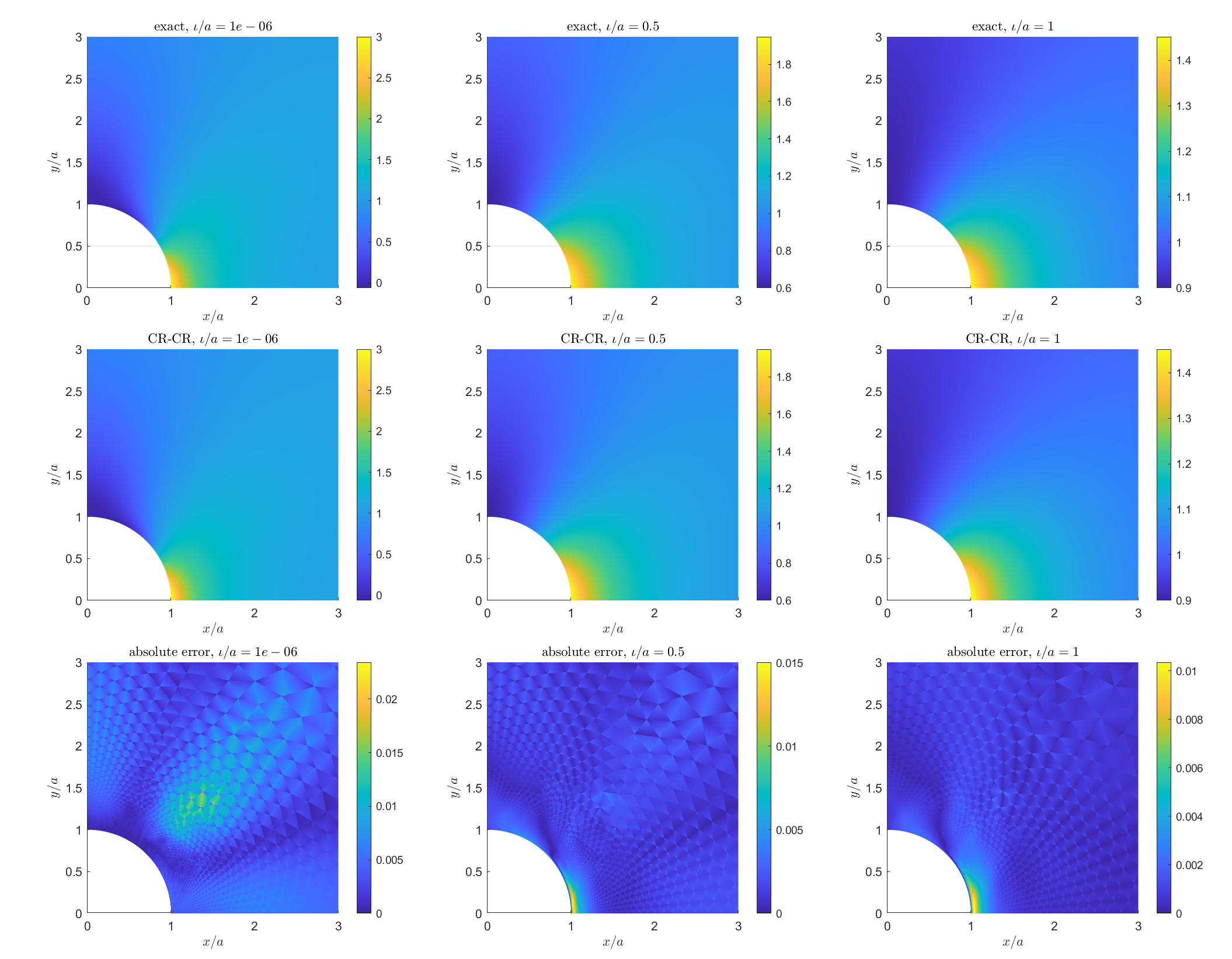}
		\caption{Normalized stress \(\sigma_{yy}/\sigma_{\infty}\) for
			\(\iota/a=10^{-6},0.5,1\): analytical and CR--CR solutions and absolute error.}
		\label{fig:kirsch-syy}
	\end{figure}
	
	We finally consider the end-loaded cantilever beam studied in
	\cite{lazopoulos2010bending}.  Let
	\[
	\Omega=(0,L)\times\left(-\frac b2,\frac b2\right)
	\times\left(-\frac c2,\frac c2\right).
	\]
	The left end is hard-clamped, and the beam is subjected to the constant
	surface force
	\begin{align*}
		u=\partial_xu=0
		\quad\text{on }\{0\}\times
		\left(-\frac b2,\frac b2\right)
		\times\left(-\frac c2,\frac c2\right),\\
		f_{\rm L}=(0,0,-f_z)^T\quad\text{on }\{L\}\times
		\left(-\frac b2,\frac b2\right)
		\times\left(-\frac c2,\frac c2\right).
	\end{align*}
	The boundary conditions on the remaining faces are all homogeneous natural boundary conditions. There are no body forces.
	
	Since \(u=0\) on the fixed planar face, all its tangential derivatives
	vanish there; consequently, \(\varepsilon(u)=0\) on this face is
	equivalent to the additional hard-clamping condition
	\(\partial_xu=0\).  We therefore introduce
	\begin{align*}
		V_h^{\rm sg,beam}
		:=
		\bigg\{
		(v_h,\psi_h)\in V_h^{\rm sg}:
		\int_Fv_h\,\mathrm ds=0,\quad
		\int_F\psi_h\,\mathrm ds=0,\\
		\forall F\in\mathcal F_h^\partial,\,
		F\subset\{0\}\times \left(-\frac b2,\frac b2\right)\times\left(-\frac c2,\frac c2\right)
		\bigg\}.
	\end{align*}
	The CR--CR approximation is the ordered pair
	\(\bm u_h=(u_h,\phi_h)\in V_h^{\rm sg,beam}\) satisfying
	\begin{equation}
		a_h^{\rm sg}(\bm u_h,\bm v_h)
		+s_h^{\rm sg}(\bm u_h,\bm v_h)
		=\int_{\{L\}\times
			\left(-\frac b2,\frac b2\right)
			\times\left(-\frac c2,\frac c2\right)}f_{\rm L}\cdot v_h\mathrm{d}s
		\quad
		\forall\bm v_h=(v_h,\psi_h)\in V_h^{\rm sg,beam}.
		\label{eq:beam-discrete}
	\end{equation}
	The constraints on the fixed end eliminate the rigid-body motions.  Hence, \(\norme{\cdot}\) is a norm on \(V_h^{\rm sg,beam}\), and \eqref{eq:beam-discrete} is well-posed.
	
	To validate the reliability of the discrete solution, we define the following free-end deflection
	\begin{equation*}
		\delta_h
		:=\left|
		\frac{1}{bc}
		\int_{\{L\}\times(-b/2,b/2)\times(-c/2,c/2)}
		u_{h}\cdot e_z\,\mathrm ds
		\right|,
	\end{equation*}
	where \(e_z\) denotes the unit vector along the positive \(z\)-axis. Correspondingly, the reference value \(\delta_{\rm L}\) can be derived from \cite[Section 4.2]{lazopoulos2010bending}, and \(\delta_{\rm P}\) can be derived from \cite[Section~6]{papargyri2003bending}; see Appendix~\ref{app:lazopoulos-reference} for the detailed derivation.
	
	For the test calculations we consider beams made of epoxy with typical properties \(E=3.5~\mathrm{GPa}\) and \(\nu =0.35\), and  take
	\[
	L=20~\mu\mathrm{m},\qquad b=c=1~\mu\mathrm{m},\qquad
	f_z=1~\mathrm{MPa}.
	\]
	The Lam\'e parameters used in the computation are therefore
	\[
	\mu=\frac{E}{2(1+\nu)}\approx1.30~\mathrm{GPa},
	\qquad
	\lambda=\frac{E\nu}{(1+\nu)(1-2\nu)}\approx3.02~\mathrm{GPa}.
	\]
	
	To examine the mesh dependence of the computed free-end displacement,
	we perform a mesh-refinement study for
	\(\iota/c=0.001,0.01,0.05\). The results are reported in
	Table~\ref{tab:beam-mesh-refinement}. We employ a family of
	quasi-uniform meshes, of which the mesh corresponding to
	\(h_{\max}/c=0.479457\) is shown in Figure~\ref{fig:beam-mesh}.
	
	\begin{table}[htbp]
		\centering
		\small
		\setlength{\tabcolsep}{8pt}
		\caption{Mesh refinement of the mean free-end displacement
			\(\delta_h~[\mu\mathrm{m}]\) and comparison with the
			Papargyri--Beskou and Lazopoulos-type reference values.}
		\label{tab:beam-mesh-refinement}
		\begin{tabular}{c|ccc}
			\hline
			\(h_{\max}/c\)
			& \(\iota/c=0.001\)
			& \(\iota/c=0.01\)
			& \(\iota/c=0.05\)\\
			\hline
			\(0.866025\) & \(10.4429\) & \(10.4420\) & \(10.4202\)\\
			\(0.702175\) & \(9.88952\) & \(9.88820\) & \(9.85640\)\\
			\(0.479457\) & \(9.53157\) & \(9.52921\) & \(9.47242\)\\
			\(0.354809\) & \(9.28798\) & \(9.28397\) & \(9.18772\)\\
			\(0.169876\) & \(9.16369\) & \(9.15652\) & \(8.98647\)\\
			\hline
			\(\delta_{\rm P}\) & \(9.14149\) & \(9.12915\) & \(9.07446\)\\
			\(\delta_{\rm L}\) & \(9.14138\) & \(9.11822\) & \(8.81112\)\\
			\hline
		\end{tabular}
	\end{table}
	
	\begin{figure}[H]
		\centering
		\includegraphics[width=0.5\textwidth]{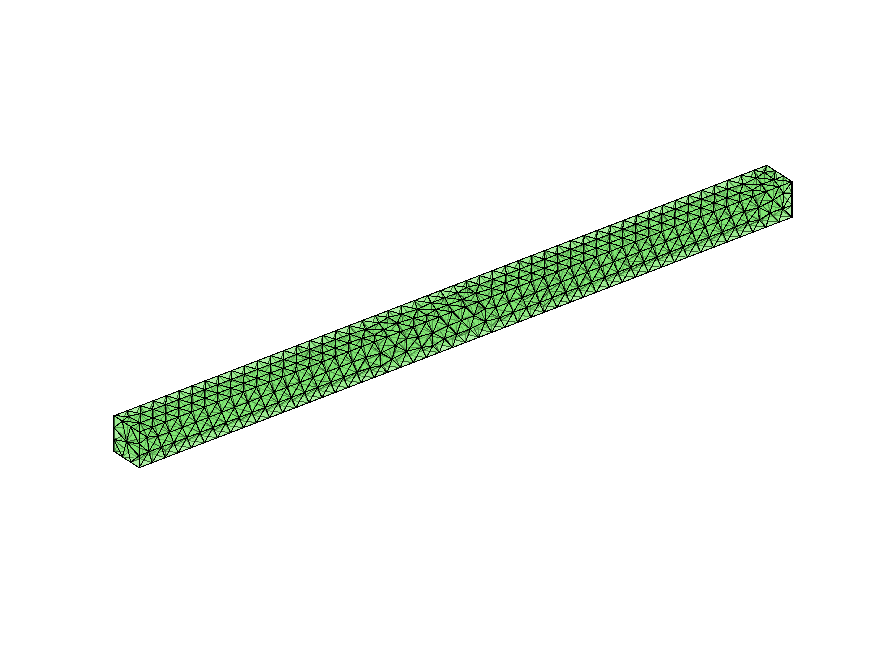}
		\caption{A quasi-uniform beam mesh with
			\(h_{\max}/c=0.479457\).}
		\label{fig:beam-mesh}
	\end{figure}
	
	For \(\iota/c=0.001,0.01,0.05\), the relative changes in the
	computed displacement between the two finest meshes are approximately
	\(1.36\%\), \(1.39\%\), and \(2.24\%\), respectively. These results
	indicate that the computed free-end displacement becomes stable under
	mesh refinement in the considered small-length-scale regime. On the
	finest mesh, the relative differences from the Papargyri--Beskou
	reference values are approximately \(0.24\%\), \(0.30\%\), and
	\(0.97\%\), respectively, whereas the corresponding differences from
	the Lazopoulos-type reference values are approximately \(0.24\%\),
	\(0.42\%\), and \(1.99\%\). Thus, for the two larger length-scale
	ratios considered here, the finest-mesh CR--CR results are closer to
	the Papargyri--Beskou reference values.
		\bibliographystyle{abbrv}
	\bibliography{ref.bib}

@article{Gallistl2017,
  author    = {Gallistl, Dietmar},
  title     = {Stable splitting of polyharmonic operators by generalized {Stokes} systems},
  journal   = {Mathematics of computation},
  year      = {2017},
  volume    = {86},
  pages     = {2555--2577},
  issn      = {0025-5718},
  doi       = {10.1090/mcom/3208},
  url       = {https://doi.org/10.1090/mcom/3208},
  publisher = {American Mathematical Society}
}

@article{Morley1968,
  author  = {Morley, L. S. D.},
  title   = {The triangular equilibrium element in the solution
             of plate bending problems},
  journal = {Aeronautical Quarterly},
  volume  = {19},
  number  = {2},
  pages   = {149--169},
  year    = {1968},
  doi     = {10.1017/S0001925900004546}
}

@article{BrennerSung2005,
  author  = {Brenner, Susanne C. and Sung, Li-Yeng},
  title   = {{$C^0$} interior penalty methods for fourth order
             elliptic boundary value problems on polygonal domains},
  journal = {Journal of Scientific Computing},
  volume  = {22--23},
  pages   = {83--118},
  year    = {2005},
  doi     = {10.1007/s10915-004-4135-7}
}

@article{ChenPaniQiu2022,
  author = {Chen, Huangxin and Pani, Amiya K. and Qiu, Weifeng},
  title = {A Mixed Finite Element Scheme for Biharmonic Equation with Variable Coefficient and von Kármán Equations},
  journal = {Communications in Computational Physics},
  year = {2022},
  volume = {31},
  number = {5},
  pages = {1434-1466},
  doi = {10.4208/cicp.OA-2021-0255},
  url = {https://doi.org/10.4208/cicp.OA-2021-0255}
}

@article{AskesAifantis2011,
  author  = {Askes, Harm and Aifantis, Elias C.},
  title   = {Gradient elasticity in statics and dynamics: An overview of formulations, length scale identification procedures, finite element implementations and new results},
  journal = {International Journal of Solids and Structures},
  volume  = {48},
  number  = {13},
  pages   = {1962--1990},
  year    = {2011},
  doi     = {10.1016/j.ijsolstr.2011.03.006}
}

@article{lazopoulos2010bending,
  author = {Lazopoulos, Konstantinos-Anastasios and Lazopoulos, Anastasios K.},
  title = {Bending and buckling of thin strain gradient elastic beams},
  journal = {European Journal of Mechanics - A/Solids},
  year = {2010},
  volume = {29},
  number = {5},
  pages = {837--843},
  doi = {10.1016/j.euromechsol.2010.04.001},
}

@article{papargyri2003bending,
  author = {Papargyri-Beskou, S. and Tsepoura, K. G. and Polyzos, Demosthenes and Beskos, Dimitri E.},
  title = {Bending and stability analysis of gradient elastic beams},
  journal = {International Journal of Solids and Structures},
  year = {2003},
  volume = {40},
  number = {2},
  pages = {385--400},
  doi = {10.1016/s0020-7683(02)00522-x},
}

@article{cai2017Discontinuous,
  author = {Cai, Zhiqiang and He, Cuiyu and Zhang, Shun},
  title = {Discontinuous finite element methods for interface problems: Robust a priori and a posteriori error estimates},
  journal = {SIAM Journal on Numerical Analysis},
  year = {2017},
  volume = {55},
  number = {1},
  pages = {400--418},
  doi = {10.1137/16m1056171},
}

@article{khakalo2017gradient,
  author = {Khakalo, Sergei and Niiranen, Jarkko},
  title = {Gradient-elastic stress analysis near cylindrical holes in a plane under bi-axial tension fields},
  journal = {International Journal of Solids and Structures},
  year = {2017},
  volume = {110--111},
  pages = {351--366},
  doi = {10.1016/j.ijsolstr.2016.10.025},
}

@article{scott1990finite,
  author = {Scott, L. Ridgway and Zhang, Shangyou},
  title = {Finite element interpolation of nonsmooth functions satisfying boundary conditions},
  journal = {Mathematics of Computation},
  year = {1990},
  volume = {54},
  number = {190},
  pages = {483--493},
  doi = {10.1090/s0025-5718-1990-1011446-7},
}

@article{li2021h2korn,
  author = {Li, Hongliang and Ming, Pingbing and Wang, Huiyu},
  title = {{$H^2$}-{Korn}'s inequality and the nonconforming elements for the strain gradient elastic model},
  journal = {Journal of Scientific Computing},
  year = {2021},
  volume = {88},
  number = {3},
  pages = {78},
  doi = {10.1007/s10915-021-01597-7},
}

@article{engel2002continuous,
  author = {Engel, Gerald and Garikipati, Krishna and Hughes, Thomas J. R. and Larson, Mats G. and Mazzei, Luca and Taylor, Robert L.},
  title = {Continuous/discontinuous finite element approximations of fourth-order elliptic problems in structural and continuum mechanics with applications to thin beams and plates, and strain gradient elasticity},
  journal = {Computer Methods in Applied Mechanics and Engineering},
  year = {2002},
  volume = {191},
  number = {34},
  pages = {3669--3750},
  doi = {10.1016/s0045-7825(02)00286-4},
}

@article{amanatidou2002mixed,
  author = {Amanatidou, E. and Aravas, Nikolaos},
  title = {Mixed finite element formulations of strain-gradient elasticity problems},
  journal = {Computer Methods in Applied Mechanics and Engineering},
  year = {2002},
  volume = {191},
  number = {15--16},
  pages = {1723--1751},
  doi = {10.1016/s0045-7825(01)00353-x},
}

@article{crouzeix1973conforming,
  author = {Crouzeix, Michel and Raviart, Pierre-Arnaud},
  title = {Conforming and nonconforming finite element methods for solving the stationary {Stokes} equations {I}},
  journal = {Revue fran{\c{c}}aise d'automatique, informatique, recherche op{\'{e}}rationnelle. Analyse num{\'{e}}rique},
  year = {1973},
  volume = {7},
  number = {R3},
  pages = {33--75},
  doi = {10.1051/m2an/197307r300331},
}

@misc{li2018new,
  author = {Li, Hongliang and Ming, Pingbing and Shi, Zhong-ci},
  title = {New nonconforming elements for linear strain gradient elastic model},
  year = {2018},
  howpublished = {arXiv preprint arXiv:1809.00819},
  eprint = {1809.00819},
  archiveprefix = {arXiv},
  primaryclass = {math.NA},
  url = {https://arxiv.org/abs/1809.00819},
}

@article{li2017two,
  author = {Li, Hongliang and Ming, Pingbing and Shi, Zhong-ci},
  title = {Two robust nonconforming {$H^2$}-elements for linear strain gradient elasticity},
  journal = {Numerische Mathematik},
  year = {2017},
  volume = {137},
  number = {3},
  pages = {691--711},
  doi = {10.1007/s00211-017-0890-x},
}

@article{zhang2009family,
  author = {Zhang, Shangyou},
  title = {A family of {3D} continuously differentiable finite elements on tetrahedral grids},
  journal = {Applied Numerical Mathematics},
  year = {2009},
  volume = {59},
  number = {1},
  pages = {219--233},
  doi = {10.1016/j.apnum.2008.02.002},
}

@article{Guzmán2012,
  author = {Guzm{\'{a}}n, Johnny and Leykekhman, Dmitriy and Neilan, Michael},
  title = {A family of non-conforming elements and the analysis of {Nitsche}'s method for a singularly perturbed fourth order problem},
  journal = {Calcolo},
  year = {2012},
  volume = {49},
  number = {2},
  pages = {95--125},
  doi = {10.1007/s10092-011-0047-8},
}

@article{Trygve2001,
  author = {Nilssen, Trygve K. and Tai, Xue-Cheng and Winther, Ragnar},
  title = {A robust nonconforming {$H^2$}-element},
  journal = {Mathematics of Computation},
  year = {2001},
  volume = {70},
  number = {234},
  pages = {489--505},
  doi = {10.1090/s0025-5718-00-01230-8},
}

@article{Gao2020,
  author = {Gao, Fu Chang and Lai, Ming Jun},
  title = {A new {$H^2$} regularity condition of the solution to {Dirichlet} problem of the {Poisson} equation and its applications},
  journal = {Acta Mathematica Sinica, English Series},
  year = {2020},
  volume = {36},
  number = {1},
  pages = {21--39},
  doi = {10.1007/s10114-019-8015-3},
}

@article{wang2018morley,
  author = {Wang, Wenqing and Huang, Xuehai and Tang, Kai and Zhou, Ruiyue},
  title = {{Morley--Wang--Xu} element methods with penalty for a fourth order elliptic singular perturbation problem},
  journal = {Advances in Computational Mathematics},
  year = {2018},
  volume = {44},
  number = {4},
  pages = {1041--1061},
  doi = {10.1007/s10444-017-9572-6},
}

@techreport{LasisSuli2003,
  author = {Lasis, Andris and S{\"{u}}li, Endre},
  title = {{Poincar{\'{e}}}-type inequalities for broken {Sobolev} spaces},
  year = {2003},
  number = {03/10},
  address = {Oxford},
  institution = {Oxford University Computing Laboratory},
  type = {Technical Report},
  url = {https://ora.ox.ac.uk/objects/uuid%3Adb8e7bdc-3336-4872-8f1c-df2838b06af7},
}

@misc{HuangTang2026SGE,
  author        = {Huang, Xuehai and Tang, Zheqian},
  title         = {A lowest-order robust mixed finite element method
                   with a third-order tensor variable for strain
                   gradient elasticity},
  year          = {2026},
  eprint        = {2607.17134},
  archivePrefix = {arXiv},
  primaryClass  = {math.NA},
  doi           = {10.48550/arXiv.2607.17134}
}

@article{HuangTang2025,
  author  = {Huang, Xuehai and Tang, Zheqian},
  title   = {Robust and optimal mixed methods for a fourth-order
             elliptic singular perturbation problem},
  journal = {Journal of Scientific Computing},
  volume  = {105},
  pages   = {72},
  year    = {2025},
  doi     = {10.1007/s10915-025-03101-x}
}

@article{Marini1985,
  author  = {Marini, Luisa Donatella},
  title   = {An inexpensive method for the evaluation of the solution
             of the lowest order {Raviart--Thomas} mixed method},
  journal = {SIAM Journal on Numerical Analysis},
  volume  = {22},
  number  = {3},
  pages   = {493--496},
  year    = {1985},
  doi     = {10.1137/0722029}
}

@article{chen2025nitsche,
  author = {Chen, Mingqing and Huang, Jianguo and Huang, Xuehai},
  title = {A nonconforming finite element method with the use of {Nitsche}'s technique for a strain gradient elastic model},
  journal = {Journal of Scientific Computing},
  year = {2025},
  volume = {104},
  number = {1},
  pages = {2},
  doi = {10.1007/s10915-025-02929-7},
}

@misc{huang2025optimal,
  author = {Huang, Jianguo and Huang, Xuehai and Tang, Zheqian},
  title = {An optimal and robust nonconforming finite element method for the strain gradient elasticity},
  year = {2025},
  howpublished = {arXiv preprint arXiv:2505.08461},
  eprint = {2505.08461},
  archiveprefix = {arXiv},
  primaryclass = {math.NA},
  url = {https://arxiv.org/abs/2505.08461},
}

@article{huang2026c0vem,
  author = {Huang, Jianguo and Yu, Yue},
  title = {A {$C^0$}-continuous nonconforming virtual element method for linear strain gradient elasticity},
  journal = {IMA Journal of Numerical Analysis},
  year = {2026},
  pages = {draf145},
  note = {Advance online publication},
  doi = {10.1093/imanum/draf145},
}

@book{ShiWang2010,
  author = {Shi, Zhong-ci and Wang, Ming},
  title = {Finite element methods},
  year = {2010},
  publisher = {Science Press},
  address = {Beijing},
  note = {in Chinese},
  isbn = {9787030264756},
}

@article{Ming2006,
  author = {Wang, Ming and Xu, Jinchao},
  title = {The {Morley} element for fourth order elliptic equations in any dimensions},
  journal = {Numerische Mathematik},
  year = {2006},
  volume = {103},
  number = {1},
  pages = {155--169},
  doi = {10.1007/s00211-005-0662-x},
}

@article{argyris1968tuba,
  author = {Argyris, John H. and Fried, Isaac and Scharpf, Dieter W.},
  title = {The {TUBA} family of plate elements for the matrix displacement method},
  journal = {The Aeronautical Journal},
  year = {1968},
  volume = {72},
  number = {692},
  pages = {701--709},
  doi = {10.1017/s000192400008489x},
}

@article{brenner2004korn,
  author = {Brenner, Susanne Cecelia},
  title = {{Korn}'s inequalities for piecewise {$H^1$} vector fields},
  journal = {Mathematics of Computation},
  year = {2004},
  volume = {73},
  number = {247},
  pages = {1067--1087},
  doi = {10.1090/s0025-5718-03-01579-5},
}

@book{maz_i_a2010elliptic,
  author = {Maz'ya, Vladimir G. and Rossmann, J{\"{u}}rgen},
  title = {Elliptic equations in polyhedral domains},
  year = {2010},
  series = {Mathematical Surveys and Monographs},
  volume = {162},
  publisher = {American Mathematical Society},
  address = {Providence, RI},
  doi = {10.1090/surv/162},
  isbn = {9780821849835},
}

@article{brenner1992linear,
  author = {Brenner, Susanne Cecelia and Sung, Li-Yeng},
  title = {Linear finite element methods for planar linear elasticity},
  journal = {Mathematics of Computation},
  year = {1992},
  volume = {59},
  number = {200},
  pages = {321--338},
  doi = {10.1090/s0025-5718-1992-1140646-2},
}

@article{aifantis1992role,
  author = {Aifantis, Elias C.},
  title = {On the role of gradients in the localization of deformation and fracture},
  journal = {International Journal of Engineering Science},
  year = {1992},
  volume = {30},
  number = {10},
  pages = {1279--1299},
  doi = {10.1016/0020-7225(92)90141-3},
}

@article{chen2026robust,
  author = {Chen, Mingqing and Huang, Jianguo and Huang, Xuehai},
  title = {A robust {$C^0$} interior penalty method for a gradient-elastic {Kirchhoff} plate model},
  journal = {Journal of Scientific Computing},
  year = {2026},
  volume = {106},
  number = {3},
  pages = {79},
  doi = {10.1007/s10915-026-03198-8},
}

@article{chen2023robust,
  author = {Chen, Mingqing and Huang, Jianguo and Huang, Xuehai},
  title = {A robust lower-order mixed finite element method for a strain gradient elastic model},
  journal = {SIAM Journal on Numerical Analysis},
  year = {2023},
  volume = {61},
  number = {5},
  pages = {2237--2260},
  doi = {10.1137/22m151649x},
}

@article{riesselmann2024efficient,
  author = {Riesselmann, Johannes and Rezaei, Nasrin and Placidi, Luca and Balzani, Daniel},
  title = {An efficient mixed finite element formulation for {3D} strain gradient elasticity},
  journal = {Computer Methods in Applied Mechanics and Engineering},
  year = {2024},
  volume = {432},
  pages = {117422},
  doi = {10.1016/j.cma.2024.117422},
}

@article{liao2023taylor,
  author = {Liao, Yulei and Ming, Pingbing and Xu, Yun},
  title = {{Taylor--Hood} like finite elements for nearly incompressible strain gradient elasticity problems},
  journal = {Journal of Scientific Computing},
  year = {2023},
  volume = {95},
  number = {1},
  pages = {4},
  doi = {10.1007/s10915-023-02135-3},
}

@article{liao2019family,
  author = {Liao, Yulei and Ming, Pingbing},
  title = {A family of nonconforming rectangular elements for strain gradient elasticity},
  journal = {Advances in Applied Mathematics and Mechanics},
  year = {2019},
  volume = {11},
  number = {6},
  pages = {1263--1286},
  doi = {10.4208/aamm.oa-2018-0268},
}

@article{wriggers2023virtual,
  author = {Wriggers, Peter and Hudobivnik, Bla{\v{z}}},
  title = {Virtual element formulation for gradient elasticity},
  journal = {Acta Mechanica Sinica},
  year = {2023},
  volume = {39},
  number = {4},
  pages = {722306},
  doi = {10.1007/s10409-022-22306-x},
}

@article{hosseini20223d,
  author = {Hosseini, Seyyed Bahram and Niiranen, Jarkko},
  title = {{3D} strain gradient elasticity: Variational formulations, isogeometric analysis and model peculiarities},
  journal = {Computer Methods in Applied Mechanics and Engineering},
  year = {2022},
  volume = {389},
  pages = {114324},
  doi = {10.1016/j.cma.2021.114324},
}
\end{document}